\documentclass[a4paper,10pt,reqno]{article}
\usepackage{amsthm}
\usepackage{amsmath,amssymb,amsfonts,amsxtra}
\usepackage{epic}
\usepackage{verbatim}
\usepackage[french,english]{babel}
\usepackage[T1]{fontenc}
\usepackage{lmodern}
\usepackage[mathscr]{euscript}
\usepackage{enumerate}
\usepackage{xcolor}
\definecolor{rouge}{rgb}{0.7,0,0}
\definecolor{bleu}{rgb}{0,0,0.7}
\usepackage[colorlinks=true,linkcolor=bleu,urlcolor=violet,citecolor=rouge]{hyperref}
\usepackage{breakurl}
\usepackage{xspace}
\usepackage{multirow}
\usepackage{longtable}

\usepackage[all]{xy}
\makeatletter
\long\def\nnfoottext#1{\insert\footins{\footnotesize
    \interlinepenalty\interfootnotelinepenalty
    \splittopskip\footnotesep
    \splitmaxdepth \dp\strutbox \floatingpenalty \@MM
    \hsize\columnwidth \@parboxrestore
   \edef\@thefnmark{}
   \edef\@currentlabel{}\@makefntext
    {\rule{\z@}{\footnotesep}\ignorespaces
      #1\strut}}}
\makeatother

\numberwithin{equation}{section}

\newtheorem{thm}{Theorem}[subsection]

\newtheorem{prop}[thm]{Proposition}
\newtheorem{lm}[thm]{Lemma}
\newtheorem{cor}[thm]{Corollary}

\theoremstyle{definition}
\newtheorem{defi}[thm]{Definition}

\theoremstyle{remark}
\newtheorem{remark}{Remark}
\newtheorem{remarks}[remark]{Remarks}

\newtheorem*{mercis}{Acknowledgments}

\numberwithin{remark}{section}

\theoremstyle{plain}

\def\ad{\operatorname {ad}}

\def\codim{\operatorname {codim}}

\DeclareMathOperator{\rk}{rk}
\def\Hom{\operatorname {Hom}}

\def\GL{\operatorname {GL}}

\DeclareMathAlphabet{\calptmx}{OMS}{ztmcm}{m}{n}

\newcommand{\K}{{\Bbbk}}
\newcommand{\Z}{\mathbb{Z}}
\newcommand{\N}{\mathbb{N}}

\newcommand{\g}{\mathfrak{g}}
\newcommand{\h}{\mathfrak{h}}

\newcommand{\kk}{\mathfrak{k}}
\newcommand{\af}{\mathfrak{a}}

\newcommand{\pp}{\mathfrak{p}}

\newcommand{\mf}{\mathfrak{m}}

\newcommand{\cc}{\mathfrak{c}}
\newcommand{\sfr}{\mathfrak{s}}
\newcommand{\wfr}{\mathfrak{w}}
\newcommand{\tf}{\mathfrak{t}}

\newcommand{\CC}{\mathfrak{C}}

\newcommand{\gl}{\mathfrak{gl}}
\newcommand{\sld}{\mathfrak{sl}_2}

\newcommand{\sln}{\mathfrak{sl}}
\newcommand{\so}{\mathfrak{so}}
\newcommand{\spn}{\mathfrak{sp}}

\newcommand{\Od}{\mathcal{O}}
\newcommand{\II}{\mathcal{I}}
\newcommand{\NN}{\mathcal{N}}

\newcommand{\pr}{\mathrm{pr}}

\newcommand{\lnq}{<}
\newcommand{\gnq}{>}

\newcommand{\cpgs}{\cc_{\pp}(\g^s)^{\bullet}}
\newcommand{\cggs}{\cc_{\g}(\g^s)^{\bullet}}

\newcommand{\cpg}[1]{\cc_{\pp}(\g^{#1})^{\bullet}}

\newcommand{\cggnb}[1]{\cc_{\g}(\g^{#1})}
\newcommand{\cpgnb}[1]{\cc_{\pp}(\g^{#1})}
\newcommand{\ckknb}[1]{\cc_{\kk}(\g^{#1})}
\newcommand{\cppsnb}{\cc_{\pp}(\pp^s)}
\newcommand{\cpgsnb}{\cc_{\pp}(\g^s)}
\newcommand{\cggsnb}{\cc_{\g}(\g^s)}

\def\preisomto{\vbox{\hbox to
               14pt{\hfill$\sim$\hfill}\nointerlineskip\vskip -0.2pt
               \hbox to 14pt{\rightarrowfill}}}

\def\prelongisomto{\vbox{\hbox to
                17pt{\hfill$\sim$\hfill}\nointerlineskip\vskip -0.2pt
                \hbox to 17pt{\rightarrowfill}}}

\usepackage{fullpage}

\begin{document}

\title{Irregular locus of the commuting variety  of reductive \\  symmetric Lie algebras and rigid pairs}
\author{{\sc Micha\"el~Bulois}
\thanks{{\url{Michael.Bulois@univ-brest.fr}}}}
\date{}
\maketitle

\nnfoottext{Universit\'e de Brest, CNRS, UMR 6205, Laboratoire de Math\'ematiques de Brest, 6 avenue Victor Le Gorgeu, CS 93837, F-29238 BREST Cedex 3}
\nnfoottext{Universit\'e Europ\'eenne de Bretagne, France}

\begin{abstract}
The aim of this paper is to describe the irregular locus of the commuting variety of a reductive symmetric Lie algebra.
More precisely, we want to enlighten a remark of Popov. In \cite{Po}, the irregular locus of the commuting variety of any reductive Lie algebra is described and its codimension is computed. This provides a bound for the codimension of the singular locus of this commuting variety. 
\cite[Remark 1.13]{Po} suggests that the arguments and methods of \cite{Po} are suitable for obtaining analogous results in the symmetric setting. 
We show that some difficulties arise in this case and we obtain some results on the irregular locus of the component of maximal dimension of the
``symmetric commuting variety''.
As a by-product, we study some pairs of commuting elements specific to the symmetric case, that we call rigid pairs. 
These pairs allow us to find all symmetric Lie algebras whose commuting variety is reducible.
\end{abstract}

\section{Introduction and Notation}
Throughout this paper, $\K$ is an algebraically closed field of characteristic $0$,
$\g$ is a finite dimensional reductive Lie algebra over $\K$ and $\theta$ is an involution of $\g$ turning $(\g,\theta)$ into 
a symmetric Lie algebra. Denoting by $\kk$, resp. $\pp$, the $\theta$-eigenspace associated to the eigenvalue $1$ (resp. $-1$), we get a vector space decomposition
$$\g=\kk\oplus \pp.$$
We will often write $(\g,\kk)$ instead of $(\g,\theta)$ to refer to the symmetric Lie algebra.
Throughout this paper, $\tf_r$ refers to any toral Lie subalgebra (\emph{i.e.} all elements of $\tf_r$ are semisimple) of $\g$ of dimension $r$.

A Cartan subspace of $(\g,\kk)$ is a maximal toral subalgebra included in $\pp$. 
The dimension of all Cartan subspaces of $(\g,\kk)$ are the same. This common dimension is called the symmetric rank of $(\g,\kk)$ and is denoted by $\rk_{sym}(\g,\kk)$.
We fix a Cartan subspace $\af$ of $(\g,\kk)$ and we embed it in a Cartan subalgebra $\h$ of $\g$.
The set of nilpotent elements of $\g$ is denoted by $\NN$.

If $H$ is an algebraic group acting rationally on a variety $V$, and $A\subset V$, we denote the stabilizer of $A$ in $H$ by 
$H^A=\{h\in H\mid h.a=a, \,\forall a\in A\}$. For $x,y\in V$, we may also write $H^x$ and $H^{x,y}$ instead of $H^{\{x\}}$ and $H^{\{x,y\}}$.
The regular locus of $A$ under the action of $H$ is denoted by $A^{\bullet}:=\{a\in A\mid \dim H^a \mbox{ is minimal}\}$
and we set $A^{irr}=A\setminus A^{\bullet}$ for the irregular locus.
Note that $A$ may not be $H$-stable, as in \eqref{cpps}.
If $V$ is irreducible, a point $x\in V$ is called smooth if the tangent space $T_x(V)$ is of dimension $\dim V$. 
The complement of the set of smooth points is called the singular locus of $V$ and is denoted by $V^{sing}$.
%

Let $G$ be the algebraic adjoint group of $\g$. We denote by $K\subset G$ the smallest subgroup with Lie algebra $\kk\cap[\g,\g]$. 
The vector space $\pp$ is stable under the action of $K$ and we consider $\pp$ as a $K$-variety.
Recall that if $\g'$ is a semisimple Lie algebra, $\g=\g'\oplus\g'$ and $\theta(x,y)=(y,x)$, then the adjoint $G'$-module $\g'$ is isomorphic to the $K$-module $\pp$. 
We will refer to this situation as the ``type 0'' case. 
If $\wfr\subset \g$ is a subspace and $A$ a subset of $\g$, the commutant of $A$ in  $\wfr$ is denoted by 
$$\cc_{\wfr}(A)=\wfr^A:=\{x\in \wfr\mid [x,A]=\{0\}\}.$$


The commuting scheme of $(\g,\kk)$ is the affine subscheme of $\pp\times\pp\cong \K^n$ defined by the ideal generated by the quadratic equations $[x,y]=0$ where $x$ and $y$ run through a basis of $\pp$. We denote it by $\mathfrak X(\g,\kk)$.
The following is a long-standing conjecture in type 0:
\begin{equation} \mbox{ $\mathfrak X(\g,\kk)$ is a reduced scheme.} \tag{$\mathcal R$}\label{conjred}\end{equation}

The commuting variety of $(\g,\kk)$ is
$$\CC(\g,\kk)=\CC_1(\g,\kk):=\{(x,y)\mid [x,y]=0\}\subset \pp\times \pp$$
sometimes abbreviated in $\CC$ or $\CC_1$ when it is clear from the context.
It can be seen as the reduced scheme of $\mathfrak X(\g,\kk)$ and we identify closed points of $\mathfrak X(\g,\kk)$ with elements of $\CC(\g,\kk)$ in the natural way.
The commuting variety is a $K$-variety for the diagonal action of $K$ on $\pp\times\pp$ defined by $k.(x,y)=(k.x,k.y)$. 
Irregularity will be considered with respect to this $K$-action.

The commuting variety is known to be irreducible in type 0, cf \cite{Ri}. 
However it has been shown by D.~Panyushev \cite{Pa94} that $\CC(\g,\kk)$ is not irreducible in the general case.
More precisely, proofs of reducibility or irreducibility of $\CC(\g,\kk)$ can be found in \cite{Pa94,Pa04,SY,PY}  for all cases but three, cf. table~\ref{table1}.
When $\CC$ is reducible few results is known about its irreducible components but it is well known 
 that $\CC$ has a unique irreducible component of maximal dimension $\dim \pp+\rk_{sym}(\g,\kk)$:
$$\CC_0(\g,\kk):=\overline{K.(\af\times\af)},$$ cf. \eqref{decompC}.
The main object considered in this paper is $\CC_0^{irr}$.

First of all, we have to determine the maximum orbit dimension in $\CC_0$. For this, we first recall that, for $(x,y)\in \pp$, we have:
$$\dim K.(x,y)=\dim K-\dim K^{(x,y)}=\dim \kk-\dim \kk^{x,y}.$$
This leads us to determine the minimum of $\dim \kk^{x,y}$ for $(x,y)\in \CC_0$. 
For any $(x,y)\in \af\times \af$ such that $\pp^{x,y}=\af$, we have $$\g^{x,y}=\g^{\af}=\af\oplus \mf$$ where $\mf$ is a reductive subalgebra of $\kk$, cf. Lemma \ref{correspondence} (iii).
Such couples are dense in $\CC_0$ and, since $\dim \kk^{x,y}$ is an upper semi-continuous function on $\CC_0$, one has
\begin{equation}\label{singloc1}\CC_0^{irr}=\{(x,y)\in\CC_0\mid \dim \kk^{x,y}\gnq \dim \mf\}.\end{equation}
For $(x,y)\in\CC$, we define the {\em irregularity number} of $(x,y)$ in $(\g,\kk)$ by \begin{equation}\label{defi} i((\g,\kk),(x,y))=\dim\kk^{x,y}-\dim\mf\end{equation}.
\begin{defi} \label{ovpr} The pair $(x,y)\in\CC$ is said to be 
\begin{itemize}\item \emph{principal} when $i((\g,\kk),(x,y))=0$ (\emph{i.e.} $\dim \kk^{x,y}=\dim \mf$),
\item \emph{semi-rigid} when $i(x,y)\leqslant 0$. 
\item \emph{rigid} when $i(x,y)=-\rk_{sym}(\g,\kk)$. 
\end{itemize}
\end{defi}
If $(\g,\kk)$ is a compact symmetric Lie algebra (\emph{i.e.} $\pp=\{0\}$) then $(0,0)\in\pp\times\pp$ is obviously a rigid pair. 
Such a rigid pair is called \emph{trivial}.
If $(x,y)$ is any rigid pair of $(\g,\kk)$, we say that $x$ (or $y$) is \emph{involved} in a rigid pair of $(\g,\kk)$. 
In this case, we also say that the $K$-orbit $K.x$ is involved in rigid pairs. 
The seek of (non-trivial) rigid pairs is related to the following conjecture: 
\begin{equation}
\label{conjBu}
\tag{$\mathcal S$}
\begin{array}{c}\mbox{The irreducible components of $\CC(\g,\kk)$ are of the form $\overline{K.(\cpgsnb+n,\pp^s)}$}\\ \mbox{where $s$ is semisimple and $n$ is involved in a rigid pair of $([\g^s,\g^s],[\kk^s,\kk^s])$}\end{array}.
\end{equation}
Note that $\CC_0$ is of this form with $s$ regular and $n=0$. We will provide in \eqref{conjBubis} another formulation of conjecture \eqref{conjBu}.

In order to study $\CC_0^{irr}$, it is convenient to work with the variety $\CC_t^{+}$ defined for $t=0,1$ by
\begin{equation}\label{cit}\CC_t^{+}(\g,\kk):=\{(x,y)\in\CC_t\mid i((\g,\kk),(x,y))\gnq 0\}.\end{equation}
Obviously, one has $\CC_0^{irr}=\CC_0^{+}\subset \CC_1^{+}$ while there is \emph{a priori} no comparison between $\CC_1^{+}$ and $\CC_1^{irr}$.
We should note here that the dimension of $\mf$ can be deduced from the dimensions of $\g$, $\kk$ and $\rk_{sym}(\g,\kk)$.
Indeed, for any generic element $x$ of $\af$, one has 
\begin{eqnarray}\dim \mf&=&(\dim \kk^x-\dim \pp^x)+\dim \pp^x\notag \\&=&\dim \kk-\dim \pp+\rk_{sym}(\g,\kk)\label{miniLevi}\\&=&2\dim \kk-\dim \g+\rk_{sym}(\g,\kk)\notag.\end{eqnarray}
We have used here \cite[Proposition 5]{KR} which states that for any $x\in\pp$
\begin{equation}\dim \kk^x-\dim \pp^x=\dim \kk-\dim \pp.\label{KRdimpx}\end{equation}

In section \ref{basics}, we parameterize the commuting variety with the help of Jordan classes. 
Then we look at $\CC_t^{+}$ and bound $\codim_{\CC_0} \CC_0^{irr}(\g,\kk)$ by some specific integers $d_t(\g,\kk)$, $t\in\{0,1\}$, defined in \eqref{defdt}.
The main result of this section is corollary \ref{codimc0JKcor}.

In section \ref{secdt}, we provide some informations on these integers $d_t(\g,\kk)$. In most cases, we are able to give their value. 
This requires some involved computations and the summary of our results can be found in the table \ref{table2} of subsection \ref{resultdt}. 
These computations also provide some examples of non-trivial rigid pairs.

In section \ref{overprincipal}, we use the previous result to give some information on the geometry of the commuting variety. 
The rigid pairs found in section \ref{secdt}  allow us to show that $\CC(\g,\kk)\neq\CC_0(\g,\kk)$ in some cases.
In this way, we get all reducible commuting varieties, ending the classification began by D.~Panyushev.
Finally, we state some properties on the singular locus of $\CC_0(\g,\kk)$ which are mostly translations of \cite{Po} 
and we prove that conjecture \eqref{conjBu} is a consequence of conjecture \eqref{conjred}.

\begin{mercis}
I would like to thank V.~Popov for encouraging me to start this work. I also thank W.~de~Graaf who kindly accepted to verify some computations using GAP4.  Finally, I thank M.~Brion and D.~Panyushev for useful conversations. 
\end{mercis}

\tableofcontents

\section{Preliminaries}
\label{basics}
\subsection{$\pp$-Levi, Satake diagrams and $\pp$-distinguished elements}
\label{pLevisdist}
For any semisimple element $s\in \pp$, we introduce the following notation $\g_{s}=[\g^{s};\g^{s}]$, $\kk_{s}=\g_{s}\cap\kk$ and $\pp_{s}=\g_{s}\cap\pp$.
Then $(\g^s,\kk^s)$ (resp. $(\g_s,\kk_s)$) is a reductive (resp. semisimple) symmetric Lie algebra and $K^s$ is a connected algebraic group acting on $\pp^s$ (resp. $\pp_s$).
Moreover, $G_1=(G^s)_{\mid \g^s}$ can be identified with the adjoint group of $\g^s$ and $K_1=(K^s)_{\mid \g^s}$ is the smallest algebraic subgroup of $G_1$ having $\kk_s$ as Lie algebra (cf \cite[24.8.5]{TY}). In particular, we can identify $K_1$-orbits and $K^s$-orbits in $\pp^s$ or $\pp_s$.

\begin{defi}
A symmetric reductive Lie subalgebra of the form $(\g^{s},\kk^{s})$ (resp. $(\g_{s},\kk_{s})$) for some semisimple element $s\in\pp$ will be called a $\pp$-Levi (resp. reduced $\pp$-Levi) of $(\g,\kk)$. 
\end{defi}
These definitions correspond to the notions of sub-symmetric and reduced sub-symmetric pairs of \cite{PY}. In \cite{Bu2}, a $\pp$-Levi is called a ``Levi factor arising from $\pp$''.

Let us give some properties of (reduced) $\pp$-Levi. 
Recall that we have a decomposition $\g^s=\cggsnb\oplus\g_s$.
For any semisimple element $s\in\pp$, one has (cf. \cite[38.8]{TY})
\begin{eqnarray}\cppsnb&=&\cggsnb\cap\pp=\cc_{\pp}(\g^s)\notag\\ 
\pp^s&=&\cpgsnb\oplus\pp_s,\label{cpps}\\
\cpgs&=& \{t\in \pp\ \mid \pp^t=\pp^s\}=\cggs\cap\pp.\notag
\end{eqnarray}

It is possible to characterize $\pp$-Levi among Levi factors with the help of Satake diagrams (cf. \cite[\S1]{PY} or \cite[\S2.2]{Bu2}). 
The Satake diagram of a symmetric Lie algebra is an analogue of the Dynkin diagram of a Lie algebra. 
Let us define it.
First, we recall that $\af$ is a Cartan subspace of $(\g,\kk)$. We embed $\af$ in a Cartan subalgebra $\h$ of $\g$. 
Consider the associated root system $\Delta(\g,\h)$ and choose a basis $B$ of $\Delta(\g,\h)$ whose associated positive roots are sent on negative roots via $\theta$, cf. \cite[2.8]{Ar}.
This gives rise to a Dynkin diagram whose nodes are the elements of $B$. 
The nodes corresponding to elements $\alpha$ of $B$ such that $\alpha_{\mid \af}=0$ are colored in black.
Other nodes are colored in white and we connect two white nodes by a two-sided arrow when they correspond to elements $\alpha,\beta\in B$ satisfying 
$\alpha_{\mid \af}=\beta_{\mid \af}$. This new diagram with arrows and colored nodes is the Satake diagram of $(\g,\kk)$ and is denoted by $S(\g,\kk)$. 
It does not depend on the choice of $\af$ or $B$ and two semisimple symmetric Lie algebras are isomorphic if and only if they have the same Satake diagram \cite[2.14]{Ar}. 

A symmetric Lie algebra $(\g,\kk)$ will be called \emph{simple} if $\g$ is semisimple and the Satake diagram (including arrows) $S(\g,\kk)$ is connected. 
It may happen that $(\g,\kk)$ is simple while $\g$ is not. This corresponds the type 0 case. 
In this case, $\g=\g'\oplus\g'$ where $\g'$ is simple and $S(\g,\kk)$ is the doubled Dynkin diagram of $\g'$
having only white nodes and where each node of the first Dynkin diagram is connected to corresponding node of the second Dynkin diagram. 
The Satake diagrams of the other simple symmetric Lie algebra are recalled in the table~\ref{table1} given in the appendix.

Recall now that for each $\pp$-Levi of the form $(\g^s,\kk^s)$ there exists an element $k\in K$ such that $t=k.s$ satisfies
\begin{itemize}
\item $t\in \af$, 
\item The set $B^t:=\{\alpha\in B\mid \alpha(t)=0\}$ is a basis of the root system $\{\alpha\in \Delta(\g,\kk)\mid \alpha(t)=0\}$. 
\end{itemize}
Such a $\pp$-Levi $(\g^{t},\kk^{t})$ is called \emph{standard} (with respect to $(\h,\af, B)$). Its associated reduced $\pp$-Levi $(\g_{t},\kk_{t})$ is also called standard.
We denote by $L(\g,\kk)$ the set of standard reduced $\pp$-Levi.
Note that we can recover a standard $\pp$-Levi from its reduced form $(\g',\kk')\in L(\g,\kk)$ via $(\g'+\h,\kk'+\h\cap\kk')$. 
Thus we will often identify the set of standard $\pp$-Levi with $L(\g,\kk)$.
In particular, if $L=(\g',\kk')$ is a standard reduced $\pp$-Levi and $(\g'',\kk'')$ its associated standard $\pp$-Levi, we have 
$$\cc_{\af}(L):=\cc_{\af}(\g')=\cc_{\pp}(\g'')\subset \af \quad \mbox{ and } \quad \cc_{\h}(L):=\cc_{\h}(\g')=\cc_{\g}(\g'')\subset \h.$$

Obviously, if $L=(\g_t,\kk_t)$ is a standard reduced $\pp$-Levi, $B^s=\{\alpha\in B\mid \alpha(s)=0\}=B^t$ does not depend on the choice of $s\in \cc_{\af}(L)^{\bullet}$ 
and $B^s$ contains the set of black nodes. 
In the same way, if $\alpha,\beta$ are two nodes connected by an arrow then we either have $\alpha,\beta\in B^s$ or 
$\alpha,\beta\notin B^s$. 
Such a subdiagram containing connected white nodes by pairs and all black nodes is called a \emph{Satake sub-diagram} of $S(\g,\kk)$. 
We have seen that $B^s$ defines a Satake sub-diagram of $S(\g,\kk)$. 
We denote it by $S(L)$ since it is isomorphic to the Satake diagram of $L$ with respect to $(\h\cap L, \af\cap L, B^s)$. 
In fact, it is easy to show that the map $L\rightarrow S(L)$ defines a bijective correspondence 
between $L(\g,\kk)$ and the set of Satake sub-diagrams of $S(\g,\kk)$ \cite[\S2.3]{Bu2}. 
This correspondence enables to read several properties of $\pp$-Levi, among which
\begin{lm}\label{correspondence}
Assume that $\g$ is semisimple and let $L\in L(\g,\kk)$.\\
(i) The symmetric rank of the reduced standard $\pp$-Levi $L$ is equal to the number of white nodes minus the number of arrows of $S(L)$.\\
(ii) $\dim \cc_{\af}(L)$ is equal to the number of white nodes minus the number of arrows of $S(\g,\kk)\setminus S(L)$.\\
(iii) There is a unique minimal standard $\pp$-Levi which is $\mf\oplus\af=\kk^{\af}\oplus\pp^{\af}$. 
It corresponds to the Satake sub-diagram of black nodes.\\ 
(iv) $L$ is a maximal proper reduced $\pp$-Levi of $(\g,\kk)$ if and only if $\dim \cc_{\af}(L)=1$.
\end{lm}

For any group $H$ for which it makes sense, we write $L_1\stackrel{H}{\sim}L_2$ for $L_1,L_2\in L(\g,\kk)$ if there exists $h\in H$ such that $h.L_1=L_2$.
The \emph{restricted Weyl group} of $(\g,\kk)$ is defined by
$$W_S=N_K(\af)/K^{\af}.$$
We refer to \cite[\S36 \& \S38]{TY} for the properties of this finite group.
Considering the action of $W_S$ on $\Delta'(\g,\h)=\{\alpha_{\mid\af}\mid \alpha\in \Delta(\g,\h)\}\setminus\{0\}$, 
we get an action on the set of $\pp$-Levi containing $\h$. For this action, one then easily sees that
$$L'(\g,\kk):=L(\g,\kk)/\stackrel{W_S}{\sim} = L(\g,\kk)/\stackrel{K}{\sim}.$$
Moreover, it is worth noticing that $L'(\g,\kk)$ is also equal to $L(\g,\kk)/\stackrel{W}{\sim}=L(\g,\kk)/\stackrel{G}{\sim}$, cf. \cite[\S2.3]{Bu2} but we will not use this last result.\\

We give bellow a classical criterion describing inclusion relations between $\pp$-Levi. 

\begin{lm} \label{inclusion}
For any semisimple elements $s,t\in\pp$,  the following conditions are equivalent for $\wfr=\g$ or $\pp$.
\item 1) $t\in\overline{K.\cpgs}$.
\item 2) There exists $k\in K$ such that $k.t\in\cpgsnb$.
\item 3) There exists $k\in K$ such that $\wfr^s\subset k.\wfr^t$. 
\end{lm}
\begin{proof}
``2) $\Leftrightarrow$ 3)'' follows from \eqref{cpps} while ``2) $\Rightarrow$ 1)''
is obvious.\\
Assume now that 1) is satisfied and let $\Psi:\pp\rightarrow\pp//K\cong \af/W_S$ be the categorical quotient map.
Then: $$\Psi(t)\in\Psi (\overline{K.\cpgs})\subset \overline{\Psi(K.\cpgsnb)}=\overline{\Psi(\cpgsnb)}$$ 
and $\Psi_{\mid\af}$ is the finite geometric quotient map with respect to the group $W_S$. 
We can assume that $s\in\af$ and we choose $k\in K$ such that $k.t\in \af$. 
Since $\cpgsnb\subset\af$ is closed, $\Psi(\cpgsnb)=\overline{\Psi(\cpgsnb)}$ so $k.t\in W_S.(\cpgsnb)$. 2) follows and this ends the proof.
\end{proof}


We now describe an analogue of Bala-Carter's Proposition 5.3 in \cite{BC}.
Let $e\in\pp$ be a nilpotent element. If $e\neq 0$, we can embed $e$ in a \emph{normal} $\sld$-triple $\sfr=(e,h,f)$. Here, normal means $e,f\in\pp$ and $h\in\kk$, cf. \cite{KR}.
We define the characteristic grading by
$\wfr=\bigoplus_{i\in \Z}\wfr(h,i)$ for $\wfr=\g,\kk$ or $\pp$ where $\wfr(h,i):=\{x\in\wfr\mid [h,x]=ix\}$. If $e=0$, we set $\wfr(h,0):=\wfr$.
Then the commutant of $e$ can be decomposed in the same way \begin{equation}\wfr^e=\bigoplus_{i\in\N}\wfr(e,i),\label{carac}\end{equation}
where $\wfr(e,i):=\wfr(h,i)\cap\wfr^e$.
If $\sfr=\langle e,h,f\rangle$,  we have $(\g(e,0),\kk(e,0))=(\g^{\sfr},\kk^{\sfr})$ and it is a reductive symmetric pair in $(\g,\kk)$.

\begin{defi}\label{pdistdefi}
Assume first that $\g$ is semisimple.\\
The integer $\rk_{sym}(\g(e,0),\kk(e,0))$ is called the defect of $e$ and is denoted by $\delta(e)$ \cite[Definition 1.4]{Bu1}.\\
The element $e\in\pp$ is said to be $\pp$-distinguished if $\pp^e\subset \NN$.\\
A $K$-orbit whose elements are $\pp$-distinguished is also called $\pp$-distinguished.\\
If $\g$ is reductive, $e\in[\g,\g]$ is said to be $\pp$-distinguished in $\g$ if it is in $[\g,\g]$.
\end{defi}
From now on, we assume that $\g$ is semisimple. According to \cite[38.10.4]{TY}, we see that an element $e$ is $\pp$-distinguished if and only if $\delta(e)=0$ (\emph{i.e.} $\pp(e,0)=\{0\}$).
Let $\af_0$ be a Cartan subspace of $\pp(e,0)=\pp^{\sfr}$. By definition, we have $\delta(e)=\dim \af_0$. 
Denoting by $\af_0^{\bullet}:=\{a\in\af_0\mid \dim K.a \mbox{ is maximal}\}$ the set of regular elements under the action of $K$, one has the following lemma.
\begin{lm}\label{BalaCarter}
If $s\in \af_0^{\bullet}$, then $(\g^s,\kk^s)$ is a minimal $\pp$-Levi containing $e$, $\cpgsnb=\af_0$ and $e$ is $\pp^s$-distinguished in $(\g^s,\kk^s)$.\\
If $(\g^t,\kk^t)$ is any other minimal $\pp$-Levi containing $e$, then there exists $k\in K^e$ such that $k.t\in\af_0^{\bullet}$ and therefore $k.(\g^{t},\kk^{t})=(\g^s,\kk^s)$.
\end{lm}
\begin{proof}
The subspace $\af_0$ is a toral subalgebra of $\g$. In particular, there exists a dense open subset $U\in \af_0$ such that $\g^x=\g^y$ for all $x,y\in U$. Moreover, the density of $U$ implies that this common centralizer is equal to $\g^{\af_0}$ and we may set $U=\af_0^{\bullet}$.
In particular, we have $\g^s=\g^{\af_0}$ and $\af_0 \subset \cpgsnb$.

Let $t$ be as in the hypothesis. 
The nilpotent element $e$ belongs to $\pp_t$ which is the odd part of a semisimple symmetric Lie algebra. Applying Jacobson-Morozov to $(\g_t,\kk_t)$, one can find $h'\in \kk_t$ and $f'\in \pp_t$ such that $\sfr'=(e,h',f')$ is a normal $\sld$-triple. 
Then, $\cc_{\pp}(\g^t)$  is a toral subalgebra of $\pp^{\sfr'}$. Since $\sfr$ and $\sfr'$ are $K^e$-conjugated, we may assume (up to $K^e$-conjugacy) that $\cc_{\pp}(\g^t)\subset\af_0\subset \cc_{\pp}(\g^s)$. Hence $\g^s=\g^{\af_0}\subset \g^t$ and minimality of $\g^t$ implies that  $\g^{s}=\g^t$. In particular, $\cc_{\pp}(\g^t)=\af_0=\cc_{\pp}(\g^s)$.

Assume now that $e$ is not $\pp_s$-distinguished. Hence, one can choose a semisimple element $t\in \pp_s\setminus\{0\}$ such that $e\in (\pp^s)^t=\pp^{s+t}\subsetneq\pp^s$. This contradicts the minimality of $\pp^s$.
\end{proof}

\subsection{Jordan classes and structure of the commuting variety}
\label{Jordan}
For any element $x\in\pp$, we denote by $x=x_s+x_n$ the Jordan decomposition of $x$ into its semisimple and nilpotent parts. 
The commutant $\g^x$ satisfies $\g^x=\g^{x_s}\cap\g^{x_n}$ and the symmetric Lie algebra $(\g^{x_s},\kk^{x_s})$ is a $\pp$-Levi of $(\g,\kk)$. 
Any element $y\in\pp^{x}$ can be decomposed in $y=y_1+y_2$ with respect to the following direct sum which is easily obtained from \eqref{cpps}.
\begin{equation}\pp^x=\cpgnb{x_s}\oplus(\pp_{x_s})^{x_n}\label{cppsbis}\end{equation} 

It is then easy to see that the following definition makes sense, cf. \cite[39.5.2]{TY}.
\begin{defi}
The Jordan $K$-class (or decomposition $K$-class) of $x$ is defined by 
$$J_K(x)=K.\{y\in \pp\mid \pp^y=\pp^x\}=K.(\cpg{x_s}+x_n)$$
\end{defi}

Jordan classes are irreducible and locally closed subset of $\pp$. In addition, they are equivalence classes so $\g$ is a disjoint union of Jordan $K$-classes.
Define 
\begin{equation}
R(\g,\kk):=\{((\g_s,\kk_s),\Od)\mid (\g_s,\kk_s)\in L(\g,\kk), \; \Od\mbox{ is a nilpotent $K^s$-orbit in $\pp_s$ }\}.\label{Rgk}
\end{equation}
Recall from section \ref{pLevisdist} that $K^s$-orbits of $\pp_s$ are precisely the $K'$-orbits of $\pp_s$ where $K'$ is the connected subgroup of the adjoint group $G'$ of $\g_s$ such that $\mathfrak{Lie}(K')=\kk_s$.
Hence $R(\g,\kk)$ is finite and we can attach to each element of $R(\g,\kk)$ a unique Jordan $K$-class. 
Moreover, any Jordan $K$-class can be obtained in such a way.
For any element $R_1\in R(\g,\kk)$, we denote by $J_K(R_1)$ its associated Jordan $K$-class.
Observe that several elements of $R(\g,\kk)$ may have the same associated Jordan $K$-class. 
In order to get a bijective correspondence, we should consider $L'(\g,\kk)$ 
instead of $L(\g,\kk)$ and quotient the set of orbit by conjugacy under $N_K(\g^s)$ but this construction will not be used. 
We will denote by $R'(\g,\kk)$ the set of Jordan $K$-classes of $(\g,\kk)$.

We can decompose the commuting variety by means of these classes. 
For any element $x=x_s+x_n\in\pp$, let $$\CC(\g,\kk)(x):=\overline{K.(\cpg{x_s}+x_n,\pp^x)}\subset\CC(\g,\kk).$$ 
It follows from the definition of Jordan classes that if $J_{K}(x)=J_{K}(y)$ then $\CC(x)=\CC(y)$. 
Hence we can define $\CC(\g,\kk)(J_{K}(x)):=\CC(\g,\kk)(x)$ and we get a decomposition \begin{equation}\CC(\g,\kk)=\bigcup_{J\in R'(\g,\kk)} \CC(\g,\kk)(J)\label{decompC}\end{equation} into a finite union of irreducible closed subsets. The dimension of $\CC(J_K(x))$ is equal to $\dim \pp+\dim \cc_{\pp}(\g^{x_s})$.
Furthermore, these subsets are distinct since $\pr_1(\CC(J))=\overline{J}$ where $\pr_1$ denotes the projection on the first variable.
In particular, the irreducible component of $\CC$ are of the form $\CC(J)$ for some $J\in  R'(\g,\kk)$ and \eqref{conjBu} can be rewritten under the form: 
\begin{equation}\mbox{ $\CC(J_K(L,\Od))$ is an irreducible component of $\CC$ iff $\Od$ is involved in rigid pairs of $L$.}\label{conjBubis}\end{equation}
Note that $\CC_0=\CC(J_K((\mf',\mf'),\{0\}))$ where $\mf':=[\mf,\mf]$. 
In other words, $\CC_0$ is the variety corresponding to the minimal reduced $\pp$-Levi and its zero orbit. The dimension of $\CC_0$ is equal to $\dim \pp+\rk_{sym}(\g,\kk)$ and is maximal among dimensions of the other varieties of the form $\CC(J_K(x))$. This proves that $\CC_0$ is the irreducible component of $\CC$ of maximal dimension. 
Since $(0,0)$ is a trivial rigid pair in $(\mf',\mf')$, we see that Conjecture \eqref{conjBubis} agree with the previous remarks.

\begin{defi}A $K$-orbit $K.x$ is said to be \emph{rigid} in $(\g,\kk)$ if it is an irreducible component of $$\pp^{(m)}:=\{x\in\pp\mid \dim \pp^x=m\}.$$
\end{defi}
In the Lie algebra case, this is equivalent to the usual definition of \cite{LS} and to the notion of \emph{originellen Orbiten} of \cite[4.2]{Bo}.

The following proposition is a well-known consequence of Richardson's argument in \cite{Ri} which leads to the irreducibility of $\CC(\g,\kk)$ when $(\g,\kk)$ is of type 0.
We provide a proof for the sake of completeness.
\begin{prop}\label{irred}
Let $R=((\g_s,\kk_s),\Od)\in R(\g,\kk)$.
If $\CC(J_K(R))$ is an irreducible component of $\CC$, then the $K^s$-orbit $\Od$ is rigid and $\pp_s$-distinguished in $(\g_s,\kk_s)$.
\end{prop}
\begin{proof}
If $x\in\pp$, it follows from the definition of Jordan $K$-classes that $J_K(x)$ is contained in $\pp^{(m)}$ where $m=\dim \pp^x$.
Therefore, each irreducible components of $\pp^{(m)}$ lies in the closure of some Jordan class.
Set $J=J_K(L,\Od)$ and fix some $n\in\Od$.

Let $y\in \pp^s$ be such that $\Od\subset \overline{J_{K^s}(y)}$ and $\dim (\pp_s)^y=\dim (\pp_s)^n$. 
We are going to show that $y\in\Od$ which implies that $\Od$ is rigid.
First, it follows from \eqref{cpps} that $\dim \pp^{s+y}=\dim \cpgsnb+\dim (\pp_s)^y=\dim \pp^{s+n}$.
Furthermore, $t+n\in \overline{t+J_{K^s}(y)}$ for all $t\in \cc_{\pp}(\g^s)$. Hence, if $J':=J_K(s+y)$, we have $J=J_K(s+n)\subset \overline{J'}$.
For any $x\in \overline{J'}$ we have $(x,x)\in \CC(J')$ so $(x,\pp^x) \cap \CC(J')\neq\emptyset$. 
It follows that $((\pr_1)_{\mid\CC(J')})^{-1}(x)$ has dimension at least $\dim \pp^y=\dim \pp^x=m$. 
But $\CC(J')\subset \CC$ so $((\pr_1)_{\mid\CC(J')})^{-1}(x)=(x,\pp^x)$ for all $x\in J$ and $\CC(J)\subset\CC(J')$.
Therefore $\CC(J)=\CC(J')$, $J=J'$ and $y\in \Od$.

As a second step, we note that $\GL_{2}.\CC(J)$ is an irreducible subvariety of $\CC$, where the group $\GL_{2}$ acts on $\CC(\g,\kk) \subset \pp\times\pp$ in the usual way. 
In particular if $\CC(J)$ is an irreducible component of $\CC$ then 
$$\sigma:\left\{\begin{array}{rcl} \pp\times\pp&\rightarrow&\pp\times\pp\\ (x,y)& \mapsto& (y,x)\end{array}\right.$$ stabilizes $\CC(J)$. 
Considering the projection $\pr_1$ on the first component, one sees that $\pp^x\subset \overline{J}$ for all $x\in J$.
Assume now that $n$ is not $\pp_s$-distinguished and fix a semisimple element $0 \neq s'\in (\pp_s)^n$. 
Set $x=s+n\in J$ and $y=s+s'+n\in \pp^x$, then  \eqref{cppsbis} gives $$\g^y=\g^{x,s'}=\cggs\oplus (\g_s)^{s',n}\subsetneq\cggs\oplus(\g_s)^n.$$ 
Thus $\dim \g^y < \dim \g^x$, which contradicts $y\in\overline{J}$. This proves that $\Od$ is $\pp_s$-distinguished.
\end{proof}

Let $x\in\pp$ and $y=y_1+y_2\in\pp^x$ be decomposed along \eqref{cppsbis}.
Since, $\cggnb{x_s}\subset \cggnb{x}$, we see that $\cggnb{x_s}\subset\g^{x,y}$ and we get
\begin{equation}\wfr^{x,y}=\cc_{\wfr}(\g^{x_s})\oplus(\wfr_{x_s})^{x_n,y_2}.\label{ssLevi}\end{equation}
when $\wfr=\g$. Since this decomposition is $\theta$-stable, it also holds for $\wfr=\kk$ or $\pp$.
Recalling that $\CC_1(\g,\kk):=\CC(\g,\kk)$, one gets
\begin{lm}\label{igk0}
(i) $(x,y)\in \CC_1(\g,\kk)$ if and only if $(x_n,y_2)\in\CC_1(\g_{x_s},\kk_{x_s})$.\\
(ii) If $(x_n,y_2)\in \CC_0(\g_{x_s},\kk_{x_s})$ then $(x,y)\in \CC_0(\g,\kk)$.\\
\end{lm}
\begin{proof}
(i) is a direct consequence of \eqref{cppsbis}.\\
(ii) follows from  $(x,y)=(x_s+x_n,y_1+y_2)\in (\cpgnb{x_s},\cpgnb{x_s})+\CC_0(\g_{x_s},\kk_{x_s})$.\\
\end{proof}

\subsection{Structure of the irregular locus} 
In this section, we  study $\CC_0^{irr}(\g,\kk)$. In particular, we adapt some results of \cite[\S3]{Po} to the symmetric Lie algebra case. Some of the results presented below are also related to \cite[\S4]{Pa94}.

If $x\in \pp$, we define for $t=0,1$: $$\CC_t(\g,\kk)(J_K(x)):=\{(y,z)\in\CC_t(\g,\kk)\mid y\in J_K(x)\}$$ sometimes abbreviated in $\CC_t(J_K(x))$. 
Observe that this definition coincides with $\CC(J_K(x))$ when $t=1$.
Then, we can write $\CC_t=\bigcup_{J\in R'(\g,\kk)} \CC_t(J)$.
In the same way, we define $\CC_t^{+}(J)=\CC_t^{+}\cap\CC_t(J)$ and we have 
\begin{equation}\label{CtJKplus}\CC_t^{+}=\bigcup_J \CC_t^{+}(J).\end{equation}
Under previous notation, we are able to prove the following reduction lemma.  Recall that the irregularity number $i((\g,\kk),(x,y))$ is defined in \eqref{defi}. 
\begin{lm}\label{igk}
Let $x\in\pp$ and $y=y_1+y_2\in\pp^x$ be decomposed along \eqref{cppsbis} then\\
(i) $i((\g,\kk),(x,y))=i((\g_{x_s},\kk_{x_s}),(x_n,y_2))$.\\
(ii) $(x,y)\in \CC_1^{+}(\g,\kk)$ if and only if $(x_n,y_2)\in\CC_1^{+}(\g_{x_s},\kk_{x_s})$.\\
(iii) If $(x_n,y_2)\in \CC_0^{+}(\g_{x_s},\kk_{x_s})$ then $(x,y)\in \CC_0^{+}(\g,\kk)$.\\
\end{lm}
\begin{proof}
(i) Lemma \ref{correspondence} (iii) shows that $\mf'=[\mf,\mf]$ is a minimal reduced $\pp$-Levi of $(\g_{x_s},\kk_{x_s})$ so the minimal 
$\pp$-Levi of $(\g_{x_s},\kk_{x_s})$ is of dimension $\dim \mf-\dim \cc_{\kk}(\g^{x_s})$. 
On the other hand \eqref{ssLevi} implies that $\dim \kk^{x,y}-\dim (\kk_{x_s})^{x_n,y_2}=\dim \ckknb{x_s}.$\\
Then, (ii) and (iii) are deduced from (i), \eqref{cit} and Lemma \ref{igk0}.
\end{proof}
\begin{remark}\label{remarkigk}
In fact, one can easily modify the proof of Lemma \ref{igk} (i) to show the following slightly more general statement.\\
Let $t\in \pp$ be a semisimple element and $x,y\in \pp^t$ such that $[x,y]=0$. 
Decompose $x=x_1+x_2$ and $y=y_1+y_2$ along $\pp^t=\cc_{\pp}(\g^t)\oplus(\pp_t)$. 
Then $i((\g,\kk),(x,y))=i((\g_{x_s},\kk_{x_s}),(x_2,y_2))$.
\end{remark}
Note that there may \emph{a priori} exists some commuting pairs 
$(x,y)$ such that $(x,y)\in \CC_0$ and $(x_n,y_2)\notin\CC_0(\g_{x_s},\kk_{x_s})$. 
This mainly explains why we work with some inequalities in the sequel. 
Hopefully, in most situations in which we use these inequalities, (iii) turns out to be an equivalence (\emph{e.g.} when 
$\CC_1(\g_{x_s},\kk_{x_s})$ is irreducible) and inequalities become equalities.

We now look at $\CC_t^{+}(J_K(x))$ for $t=0,1$.
We can write 
\begin{eqnarray}\CC_t^{+}(J_K(x))&=&K.(\cpg{x_s}+x_n, \II_t((\g,\kk) ,x))\\
\mbox{where } \II_t((\g,\kk),x)&:=&\{y\in\pp\mid i((\g,\kk),(x,y))\gnq0\mbox{ and }(x,y)\in\CC_t\}\subset \pp^x.\end{eqnarray}
Define
\begin{equation}\label{defc}c_t((\g,\kk),x):=\codim_{\pp^x}\II_t((\g,\kk),x)+\rk_{sym}(\g_{x_s},\kk_{x_s}),
\end{equation}
so that we can state the following key proposition.
\begin{prop}\label{codimc0JK}
(i) $\dim \CC_t(\g,\kk)-\dim \CC_t^{+}(\g,\kk)(J_K(x))=c_t((\g,\kk),x)$.\\
(ii) $c_1((\g_{x_s},\kk_{x_s}),x_n)=c_1((\g,\kk),x) \leqslant c_0((\g,\kk),x)\leqslant c_0((\g_{x_s},\kk_{x_s}),x_n)$.
\end{prop}
\begin{proof}
(i) We know that $\dim \CC_1=\dim \CC_0=\dim \pp+\rk_{sym}(\g,\kk)$. Therefore
\begin{eqnarray*}\dim {\CC_t}-\dim \CC_t^{+}(J_K(x))&=&\dim \pp+\rk_{sym}(\g,\kk)-\left(\dim \kk+\dim \cpg{x_s}-\dim \kk^x+\dim \II_t((\g,\kk),x)\right)\\
&=& \dim \pp^x-\dim \II_t((\g,\kk),x)+\rk_{sym}(\g,\kk)-\dim \cpgnb{x_s}.\end{eqnarray*}
(ii) Lemma \ref{igk}(ii)-(iii) shows that the following holds in $\pp^x=\cpgnb{x_s}\oplus(\pp_{x_s})^{x_n}$:
\begin{eqnarray}\II_1((\g,\kk),x)&=&\cpgnb{x_s}\times\II_1((\g_{x_s},\kk_{x_s}),x_n),\notag \\
\II_0((\g,\kk),x)&\supset&\cpgnb{x_s}\times\II_0((\g_{x_s},\kk_{x_s}),x_n).\label{Itsubpair} \end{eqnarray}
Hence, the equality and the rightmost inequality follow.
The remaining inequality $c_1((\g,\kk),x) \leqslant c_0((\g,\kk),x)$ is obvious.
\end{proof}

The integer $c_t((\g,\kk),x)$ does only depend on the Jordan class of $x$. Hence, if $J$ is a $K$-Jordan class of $(\g,\kk)$, we may define $c_t((\g,\kk), J):=c_t((\g,\kk), x)$ for any $x\in J$. In particular, if $x$ is nilpotent, the notation $c_t((\g,\kk),K.x)$ makes sense and $c_t(R)$ is well defined for any element $R\in R(\g,\kk)$, cf \eqref{Rgk}. 
Note that we may have $c_0((\g,\kk),J_K(R)) \neq c_0(R)$ if there are strict inequalities in Proposition \ref{codimc0JK}(ii). Next, we define
\begin{equation}d_t(\g,\kk):=\min\{c_t(R)\mid R \in R(\g,\kk)\}.\label{defdt}\end{equation}
The decomposition \eqref{CtJKplus} leads to the following corollary of Proposition \ref{codimc0JK}.
\begin{cor}\label{codimc0JKcor}
$$
d_1(\g,\kk)\leqslant\codim_{\CC_0} \CC_0^{+}(\g,\kk)\leqslant d_0(\g,\kk)
$$
\end{cor}
In the Lie algebra case, $\CC_1(L)=\CC_0(L)$ for each $L\in L(\g,\kk)$, so $c_0(L,\Od)=c_1(L,\Od)$ for each nilpotent orbit $\Od$ of $L$. In particular, $d_0(\g,\kk)=d_1(\g,\kk)$ in this case and the inequalities of Corollary \ref{codimc0JKcor} are equalities, cf. \cite[Lemma~3.20]{Po}. 


\section{Computation of some $d_t(\g,\kk)$}
\label{secdt}
The minima $d_t(\g,\kk)$ are taken over a finite set and this makes them quite manageable to compute.
The aim of the present section is to express these computations. The strategy goes as follows.

First of all, observe that if $\Od$ is a nilpotent orbit in $L\in L(\g,\kk)$, then $c_t(L,\Od)\geqslant \rk_{sym}(L)$. 
Therefore, in order to find the integers $d_t(\g,\kk)$, it is sufficient to compute some well-chosen $c_t(L,\Od)$ first, and then examine all small rank cases.
Contrary to the Lie algebra case, we do not have any uniform upper bound given by the only rank one Lie algebra.
In fact, there is an infinite number of symmetric Lie algebras of symmetric rank one and we may find some arbitrary large $d_t(\g,\kk)$. 
Nevertheless, computing the rank one case gives a first estimation and provide the number $10$ as a uniform upper bound of $d_t(\g,\kk)$ for all $(\g,\kk)$ such that $\rk_{sym}(\g,\kk)\geqslant2$.
This reduces the problem to compute $c_t(L,z)$ for all semisimple symmetric Lie algebras $L$ such that $\rk_{sym} L\leqslant 10$. 
However, computing them all looks difficult and we provide some shortcut lemmas in the second subsection in order to reduce the difficulty.

In a third subsection, we explicit some special $c_t((\g,\kk),z)$ with $\rk_{sym}(\g,\kk)\geqslant 2$ improving our upper bound for $d_t(\g,\kk)$.
We also point out some rigid pairs which will be of importance in section \ref{overprincipal}.

Finally, we give a lower bound for $d_t(\g,\kk)$. This lower bound is equal to our upper bound in a significant amount of cases. 
This gives $\codim_{\CC_0} \CC_0^{+}(\g,\kk)$ in these cases, thanks to corollary \ref{codimc0JKcor}. Our bounds are summarized in table \ref{table2}.

\subsection{Rank one case}
\label{rankonecase}
In this subsection , we assume that $(\g,\kk)$ has a symmetric rank equal to one.
Note that if $K.z$ is regular then $c_t((\g,\kk),K.z)=+\infty$, so that we may forget it in the computation of $d_t(\g,\kk)$. 

\begin{lm}\label{lmrk1}
\item (i) $\CC_0=\{(t_1 x,t_2 x)\mid x\in\pp, t_1,t_2\in\K\}$,
\item (ii) $\dim\CC_0(\g,\kk)=\dim \pp+1$,
\item (iii) $c_t((\g,\kk),0)=\codim_{\pp} \pp^{irr}+1$ for $t=0,1$.
\end{lm}
\begin{proof}
Since any Cartan subspace can be written as $\K s$ for some semisimple element $s\in\pp$, 
the first two assertions are straightforward. Some details can be found in \cite[\S 4]{Pa04}.\\
The third assertion follows from $\II_t((\g,\kk),0)=\pp^{irr}$.
\end{proof}

Since non-zero semisimple elements are regular, $\pp^{irr}$ is a non-empty union of nilpotent orbits.
If the only non-regular orbit is $\{0\}$, then lemma \ref{lmrk1}(iii) shows that
\begin{equation}\label{lmrk1bis}d_t(\g,\kk)=c_t((\g,\kk),0)=\dim \pp+1.\end{equation}
For non-zero and non-regular orbits we have:
\begin{lm}\label{uniqueGorb}
There exists at most one $G$-orbit $\Od\neq\{0\}$ whose intersection with $\pp$ is a non-empty union of non-regular $K$-orbits.\\
In particular, if $\Od$ exists, $\pp^{irr}=\overline{\Od\cap\pp}=(\Od\cap\pp)\cup\{0\}$ is equidimensional of dimension $\dim \Od/2$.
\end{lm}
\begin{proof}
Let $h_1$ (resp. $h_2$) be a characteristic of any regular nilpotent element of $\pp$ (resp.  any element non-regular element $z\neq 0$ of $\pp$). 
Then $G.h_i$ intersects $\pp$ for $i=1,2$ (cf. \cite[Theorem~1]{An}) and since $\rk_{sym} (\g,\kk)=1$, there exists $\lambda\in \K$ such that
$\lambda G.h_2=h_1$. But, $h_1$ and $h_2$ are characteristics and $h_1$ is even, hence $\lambda=2$. 
Since a nilpotent $G$-orbit is uniquely determined by the $G$-orbit of its characteristic, the first claim follows.\\
The claim concerning the dimension in the second assertion is a direct consequence of \eqref{KRdimpx}.
\end{proof}
If $\Od$ is as in lemma \ref{uniqueGorb},  and $\Od'$ is a $K$-orbit in $\Od\cap\pp$, then $\Od'$ will be called \emph{subregular}.

If a subregular orbit $\Od'$ exists, the following phenomenon about the irregularity number of a commuting pair $(z,y)$ with $z\in \Od'$ occurs. 
\begin{prop}\label{miracle}
$$i((\g,\kk),(z,y))=\left\{\begin{array}{l l} -1 &\mbox{ if }y\notin \K z \\ \codim_{\pp} \pp^{irr}-1& \mbox{ if }y\in \K z. \end{array}\right.$$ 
\end{prop}
If $y\in \K z$, \eqref{miniLevi} yields \begin{eqnarray*}\dim \kk^{z,y}-\dim \mf&=&\dim \kk^{z}-(\dim \kk-\dim \pp+\dim \af) \\&=&\dim \pp-\dim K.z-1.\end{eqnarray*}
The uniform result for $y\notin \K z$ was quite unexpected by the author and will provide some applications to the study of the commuting variety.
There are only 3 types of symmetric Lie algebras of rank one in which such non-zero subregular $K$-orbit exists.
They are: $(\spn_{2n+2},\spn_{2n},\spn_2)$ $(n\geqslant 2)$ of type {\bf CII}; $(\mathfrak f_4, \so_9)$ of type {\bf FII} 
and $(\sln_{n+1},\sln_{n}\oplus \tf_1)$ $(n\geqslant 2)$ of type {\bf AIII}, cf. table \ref{table1}.
We are going to give the key points of the computations leading to the Proposition \ref{miracle}.

\begin{proof}[(Sketch of the) proof of Proposition \ref{miracle}]
Let $z$ be an element of a subregular orbit and embed it in a normal $\sld$-triple $(z,h,z')$.
This leads to the $\theta$-stable characteristic graduations $\wfr=\bigoplus_{i\in\Z} \wfr(h,i)$ and 
$\wfr^z=\bigoplus_{i\geqslant 0}\wfr(z,i)$ for $\wfr=\g$, $\kk$ or $\pp$ as in \eqref{carac}.
It is easy to see that the element $z$ is $\pp$-distinguished \cite{Pa04} so $\pp(z,0)=\{0\}$.

Next, in each case, we find $\g(h,i)=0$ for $i\geqslant 3$ and $\pp(z,2)=1$, therefore $\pp(h,2)=\pp(z,2)=\K.z$, $\pp(h,1)=\pp(z,1)$ and $\kk(h,1)=\kk(z,1)$.
In particular $\pp^z=\pp(z,1)\oplus\K z$.
Computations give also that the $\kk(z,0)$-module $\pp(z,1)$ is isomorphic to the $\sln_{n-1}\oplus\tf_1$-module $\K^{n-1}$ in type AIII, 
to the $\spn_{2n-2}\oplus\tf_1$-module $\K^{2n-2}$ in type CII, and to the $\sln_4$-module $\K^4$ in type FII where $\tf_1$ acts on $\K^l$ by multiplication on all simple factors.
It follows that for any $y\in \pp(z,1)$, we have $$[\kk(z,0);y]=\pp(z,1).$$
We also find that the pairing given by the Lie bracket $\kk(z,1)\times \pp(z,1)\cong \K^l\times\K^l\rightarrow\pp(z,2)=\K.z$ is a non-degenerate bilinear form.
Then for any $y\in \pp(z,1)\setminus\{0\}$, we get $[\kk^z;y]=\pp^z$ and for any $\lambda\in \K^{\times}$
$$\dim \kk^{z,y+\lambda z}=\dim \kk^{z,y}=\dim \kk^z-\dim \pp^z=\dim \mf-1,$$
where the last equality follows from \eqref{miniLevi}.
\end{proof}

\begin{cor}\label{cormiracle}
\item (i) A commuting pair $(x,y)$ belongs to the maximal component $\CC_0$ if and only if $\dim \kk^{x,y}\geqslant \dim \mf$. 
\item (ii) If $z\in \pp\setminus\{0\}$ belongs to a subregular $K$-orbit then $$c_0((\g,\kk),K.z)=c_1((\g,\kk),K.z)=\dim \pp^z=\codim_{\pp} \pp^{irr}.$$
\end{cor}
\begin{proof}
(i) The \emph{if} part follows from Proposition \ref{miracle} and Lemma \ref{lmrk1}(i), while the \emph{only if} part has been noticed before \eqref{singloc1}. 
\\
(ii) Proposition \ref{miracle} gives $\II_t((\g,\kk),z)=\K.z$. In particular, $c_t((\g,\kk),z)=\dim \pp^z-1+1$.
\end{proof}

We are now able to compute $d_t(\g,\kk)$ for each $(\g,\kk)$ of symmetric rank one, thanks to \eqref{lmrk1bis} and Corollary \ref{cormiracle}.

\begin{table}[h]
\center
\caption{\label{table3} The integers $d_t(\g,\kk)$ when $\rk_{sym}(\g,\kk)=1$}
\begin{tabular}{|c| c| c| c| c| c|c|c|c|c|}
\hline
type &A0 (cf. \cite{Po}) & AI& AII& \multicolumn{2}{c|}{AIII}& BDI& \multicolumn{2}{c|}{CII}& FII\\
\hline
$(\g,\kk)$& $(\sld\oplus\sld,\sld)$ & $(\sld,\so_2)$ & $(\sln_4,\spn_4)$& \multicolumn{2}{c|}{$(\sln_{p+1},\sln_p\oplus \tf_1)$}
&$(\so_{p+1},\so_p)$& \multicolumn{2}{c|}{$(\spn_{2p+2},\spn_{2p}\oplus\spn_2)$} & $(\mathfrak f_4,  \so_9)$\\
\hline
condition & /& /& /& $p=1$&$p\geqslant 2$& /  & $\;\;p=1\;\;$ & $p\geqslant 2$ & /\\
\hline
$d_t(\g,\kk)$& $4$ & $3$ & $6$ & $3$ &$p$ & $p+1$ & $5$ & $2p-1$ & $5$\\
\hline
\end{tabular}
\end{table}

\subsection{Reduction lemmas}
\label{shortcut}
If $\g$ is simple and $\rk_{sym} (\g,\kk)\geqslant2$ then it follows from table \ref{table1} that there exists a simple reduced $\pp$-Levi 
$(\g',\kk')$ of $(\g,\kk)$ isomorphic to one of the following rank-one symmetric subalgebra
$$(\sld\oplus\sld,\sld), \, (\sld,\so_2), \, (\sln_4,\spn_4), \, (\sln_6,\sln_5\oplus\tf_1), \, (\so_{10},\so_9).$$
Since, $c_t((\g',\kk'),z)\geqslant \rk_{sym}(\g',\kk')$ it is therefore sufficient to get all $c_t((\g',\kk'),z)$ with $\rk_{sym}(\g',\kk')\leqslant 10$. 
In fact, a more precise study shows that we can replace $10$ by $6$ but this is still not manageable by hand computations. 
This is why we list here some shortcuts in order to deal with less cases.

The next lemma shows that one can reduce to the case of simple symmetric Lie algebras.
\begin{lm}\label{trick1}
If $(\g',\kk')=(\g_1,\kk_1)\oplus(\g_2,\kk_2)$ and $z=z_1+z_2$ is a nilpotent element of $\pp'$ then  we have
$$c_t((\g',\kk'),z)\geqslant \min \left(c_t((\g_1,\kk_1),z), c_t((\g_1,\kk_1),z)\right) \mbox{ for $t=0,1$.}$$
In particular, $d_t(\g,\kk)=\min\{c_t(R)\mid R \in R_1(\g,\kk)\}$ where 
$$R_1(\g,\kk)=\{(L,\Od)\in R(\g,\kk) \mid L \mbox{ is simple and }\Od \mbox{ is not regular}\}.$$
\end{lm}
\begin{proof}
Let $y=y_1+y_2\in\pp_1^{z_1}\oplus\pp_2^{z_2}$. One clearly has 
$i((\g',\kk'),(z,y))=i((\g_1,\kk_1),(z_1,y_1))+i((\g_2,\kk_2),(z_2,y_2))$.
Therefore, the condition $i((\g',\kk'),(z,y))\gnq 0$ implies that there exists $j\in\{1,2\}$ such that $i((\g_j,\kk_j),(z_j,y_j))\gnq 0$.
In particular, $$\II_t((\g',\kk'),z)\subset\left(\pp_1^{z_1}\times \II_t((\g_2,\kk_2),z_2)\right)\cup\left(\II_t((\g_1,\kk_1),z_1)\times \pp^{z_2}\right)$$
and the first assertion follows.\\
Since $c_t(L,\Od)=+\infty$ when $\Od$ is regular in $L$, we can omit this case in the computation of $d_t(\g,\kk)$.
\end{proof}

If the nilpotent element $z$ is not $\pp$-distinguished, cf. Definition \ref{pdistdefi}, the following lemma will help us to find some lower bounds for $c_t((\g,\kk),z)$.
\begin{lm}\label{pdist}
Assume that $z$ is a nilpotent element such that $\delta(z)\neq 0$. Then $z$ satisfies at least one of the three following assertions.
\item (i) $\delta(z)=1$, $\codim_{\pp^z}\II_1((\g,\kk),z)=1$ and each irreducible component of $\II_1((\g,\kk),z))$ of codimension $1$ in $\pp^z$ is an irreducible component of 
$\NN\cap\pp^z$. 
\item (ii) $\codim_{\pp^z} \II_1((\g,\kk),z)\geqslant 2$.
\item (iii) There exists a proper reduced $\pp$-Levi $(\g_t,\kk_t)$ such that $z\in \pp_t$ and 
$$\codim_{\pp^z} \II_1((\g,\kk),z)\geqslant  \codim_{(\pp_t)^z} \II_1((\g_t,\kk_t),z).$$
\end{lm}
\begin{proof}
If $t$ is any semisimple element of $\pp^z\setminus\{0\}$, then $(\g_t,\kk_t)$ is a proper reduced $\pp$-Levi  such that $z\in \pp_t$.
For each $y\in (\pp_t)^z$, $y+t$ is an element of $\pp^z$ and we get from  remark \ref{remarkigk} that \begin{equation}\label{inducss}i((\g,\kk),(y+t,z))=i((\g_t,\kk_t),(y,z))=i((\g_t,\kk_t),(z,y)).\end{equation}

If $\codim_{\pp^z} \II_1((\g,\kk),z)=0$ then  $i((\g_t,\kk_t),(z,y))>0$ for all $y\in (\pp_t)^z$. 
Therefore $\codim_{(\pp_t)^z} \II_1((\g_t,\kk_t),z)~=~0$ and (iii) holds in this case.

From now on, assume that $\codim_{\pp^z} \II_1((\g,\kk),z)=1$. 
Let $Y$ be an irreducible component of $\II_1((\g,\kk),z)$ of codimension $1$ in $\pp^z$.
From \cite[Proposition 1.13]{Ha}, there exists a polynomial $f$ on $\pp^z$ such that $Y=\mathcal V_{\pp^z}(f):=\{x\in \pp^z\mid f(x)=0\}$. 
Since $\II_1((\g_t,\kk_t),z)+t=\II_1((\g,\kk),z)\cap((\pp_t)^z+t)$ \eqref{inducss}, 
we get $\mathcal V_{(\pp_t)^z+t}(f)=Y\cap((\pp_t)^z+t)\subset \II_1((\g_t,\kk_t),z)+t$.

If there exists a semisimple element $t\in \pp^z\neq\{0\}$ such that $Y\cap((\pp_t)^z+t)\neq\emptyset$, then $Y\cap((\pp_t)^z+t)$ is of codimension at most $1$ in $((\pp_t)^z+t)$
and so does $\II_1((\g_t,\kk_t),z)+t$. In this case, (iii) holds.

Assume now that $Y\cap((\pp_t)^z+t)=\emptyset$ for all semisimple element $t\in \pp^z\setminus\{0\}$. 
This implies that $Y\subset\NN\cap\pp^z$. Since $\codim_{\pp^z} \NN\cap\pp^z=\delta(z)$, cf.\cite{Bu1},
we must have $\delta(z)=1$. Recalling that $\mathcal I_1((\g,\kk),z)$ is closed, we get (i).
%
%
%
\end{proof}

\begin{cor} \label{pdistbis} We may omit $c_1(L,K.z)$ in the computation of $d_1(\g,\kk)$ 
if we have found $R\in R_1(\g,\kk)\setminus\{(L,K.z)\}$ such that
\begin{eqnarray}
\!\!\!\!c_1(R)\leqslant \rk_{sym}(L)+1&&\mbox{and $z$ is not $\pp$-distinguished in $L$,}\notag\\
\!\!\!\!c_1(R)\leqslant \rk_{sym}(L)+2&&\mbox{and }\delta(z)\geqslant 2,\notag \\
\!\!\!\!c_1(R)\leqslant \rk_{sym}(L)+2,&&\delta(z)=1,\, \NN\cap\pp^z 
\mbox{ irreducible and }\exists \,n\in \NN\cap\pp^z \mbox{ s.t. } i((\g,\kk),(z,n))>0.\notag
\end{eqnarray}
\end{cor}
\begin{proof}
Write $L=(\g',\kk')$ and assume that $\delta(z)\neq0$.
If $z$ satisfies (iii) of Lemma \ref{pdist}, we have $$c_1((\g'_t,\kk'_t),z)=\codim_{(\pp'_t)^z} \II_1((\g'_t,\kk'_t),z)+\rk(\g'_t,\kk'_t)<
\codim_{(\pp')^z} \II_1((\g',\kk'),z)+\rk(\g',\kk')=c_1((\g',\kk'),z).$$
Therefore $c_1((\g',\kk'),z)$ is not minimal and we may omit it in the computation of $d_1(\g,\kk)$.

From now on, we assume that $z$ does not satisfies (iii) of Lemma \ref{pdist}.

So $z$ satisfies (i) or (ii) and we have $\codim_{(\pp')^z} \II_1((\g',\kk'),z)\geqslant 1$. 
In particular $c_1((\g',\kk'),z)\geqslant 1+\rk_{sym}(L)\geqslant c_1(R)$ and the first assertion follows.

Similarly, the conditions of the second and third assertions force $z$ to satisfy (ii) of Lemma \ref{pdist}. 
This provides $c_1((\g',\kk'),z)\geqslant 2+\rk_{sym}(L)\geqslant c_1(R)$.
\end{proof}


\subsection{Some particular cases}
\label{specialcase}
We provide here some important specific cases which will be of importance in the next subsections.
When referring to a nilpotent orbit of a symmetric Lie algebra, we will often provide a characteristic object attached to this orbit.
This object may be a partition (case AI, AII, cf. \cite{Oh1}), an $ab$-diagram (other classical cases, cf. \cite {Oh2}) or the number of the orbit given in the works of Djokovic in exceptional cases (cf. \cite{Dj1,Dj2}). Recall that the table \ref{table1} given in the appendix provides the type of each symmetric Lie algebra.

In order to estimate $d_t(\g,\kk)$, we begin by an explicit computation in two particular cases.

The first of these two cases is $(\g,\kk)\cong(\spn_{8},\spn_4\oplus\spn_4)$ of {\bf type CII}. We assume that $z$ belongs to the $K$-orbit having  
$\begin{smallmatrix} a&b \\ b&a \\a & b\\  b &a \end{smallmatrix}$ as $ab$-diagram, cf. \cite{Oh2}.\\
Let $V=\K^8$ and let $(v_i)_{i\in[\![1,8]\!]}$ be a basis of $V$. Let $T,J\in\gl(V)$ be defined by 
$$T.v_i=(-1)^iv_{9-i}; \qquad J.v_i=\left\{\begin{array}{l l}(-1)^{i+1}v_i& \mbox{ if $i\in [\![1,4]\!],$}\\ (-1)^{i}v_i& \mbox{ if $i\in [\![5,8]\!].$} \end{array}\right.$$
The pairing $(v,w)\rightarrow {}^tvTw$ defines a skew-symmetric bilinear form, $J^2=Id$ and ${}^tJTJ=T$. 
Then $\g:=\{x\in \gl(V)\mid {}^txT+Tx=0\}\cong \spn_{8}$, $\kk:=\{x\in \g \mid x=JxJ^{-1}\}$ and  $\pp:=\{x\in \g \mid x=-JxJ^{-1}\}$
form a symmetric Lie algebra isomorphic to $(\spn_{8},\spn_4\oplus\spn_4)$, cf. \cite[Theorem 3.4]{GW}. 
In this case \eqref{miniLevi} provides $\dim \mf=40-36+2=6$ .
Define $z\in\pp$ by $z.v_i=\left\{\begin{array}{l l} v_{i-1} &\mbox{ if $i$ is even}\\ 0 &\mbox{ if $i$ is odd} \end{array}\right.$.
For any $y\in\pp^z$, one finds that $\dim \kk^{z,y}\geqslant 7$ if and only if $y.v_5=y.v_1=0$ and it is then easy to see that 
$$\codim_{\pp^z} \II_t((\g,\kk), z)=2; \qquad c_t((\g,\kk),K.z)=4.$$
Note that, since $\CC(\g,\kk)$ is irreducible in this case, the result does not depend on $t=0$ or $t=1$.

The second case is $(\g,\kk)=(\mathfrak g_2,\sln_2\oplus\sln_2)$  ({\bf type GI}) and $K.z$ is the orbit having number $4$ in the classification of Djokovic 
\cite[Table VI]{Dj1}. This orbit will be denoted by $\mathcal O_4$. 
We embed $z$ in a normal $\sld$-triple and, using the notation \eqref{carac}, it follows from \cite{JN} that $\pp^z=\pp(z,2)$, $\kk^z=\kk(z,4)$ and $\g(h,6)=\{0\}$. 
In particular $[\kk^z, \pp^z]=\{0\}$ and for any $y\in\pp^z$ we have $\kk^{z,y}=\kk^z$.
Since $z$ is not regular and $\CC(\g,\kk)$ is irreducible, we get that $$\pp^z= \II_t((\g,\kk),z)\mbox{ and }c_t((\g,\kk),K.z)=\rk_{sym}(\g,\kk)=2.$$


\smallskip

While performing the computation of some $c_t$, the author found several nontrivial rigid pairs. 
We  now give four examples of symmetric Lie algebras of rank greater than $2$ having such pairs.
The last three examples were checked by W.~de Graaf using GAP4.

The first one lies in $(\g,\kk)=(\sln(V),\sln(V_a)\oplus\sln(V_b)\oplus \tf_1)$ of {\bf type AIII} where $V=V_a\oplus V_b$, $\dim V_a=n_a$ and  $\dim V_b=n_b$. The one-dimensional toral subalgebra $\tf_1$ is generated by the diagonal semisimple element defined by $n_b Id$ (resp. $-n_a Id$) on $V_a$ (resp. $V_b$). In this case, $\pp=\Hom(V_a,V_b)\oplus\Hom(V_b,V_a)$.
Let $\epsilon\in \{0,1\}$. 
We assume that $n_b=(l+\epsilon)(l+1)$ for some $l\in \N$ and that $n_a=n_b+l+1+r$ with $r\geqslant0$. 
By assumption, one can choose a basis $(v_j^i)_{(i,j)\in B}$ of $V_b$ and a linearly independent family $(w_j^i)_{(i,j)\in A}$ of $V_a$ with
$B=\{(i,j)\mid i\in [\![2, 2l+1+\epsilon]\!], j\in [\![1,\lfloor \frac i 2\rfloor]\!]\}$ and $A=\{(i,j)\mid i\in [\![1, 2l+1+\epsilon]\!], j\in [\![1,\lceil \frac i 2\rceil]\!]\}$ 
where $\lfloor.\rfloor$ (resp. $\lceil.\rceil$) is the floor (resp. ceiling) function. 
We complete $(w_j^i)_{(i,j)\in A}$ into a basis of $V_a$ with vectors $w'_j$,  $j\in[\![1,r]\!]$.
Define nilpotent elements $z$ and $y$ by
$$\left\{\begin{array}{l l}
z.v_j^i=w_j^i\\
z.w_j^i=v_{j-1}^i \mbox{ for $j\neq 1$}\\
z.w_1^i=0\\
z.w'_j=0\\\end{array}\right.,\qquad
\left\{\begin{array}{l l}
y.v_j^i=w_j^{i-1}\\
y.w_j^i=v_{j-1}^{i-1} \mbox{ for $j\neq 1$}\\
y.w_1^i=0\\
y.w'_j=0\\\end{array}\right..$$

Pictorially: we place alternative elements of $V_a$ and $V_b$ in a Young tableau with a triangular shape of size $2l+1+\epsilon$, $z$ acts horizontally and $y$ vertically. 
For example, in the case $l=1$, $\epsilon=1$,  and $r=0$, the elements $z,y\in\pp$ have the same $ab$-diagram 
$\begin{smallmatrix}b& a &b &a\\a&b&a\\ b&a\\a\end{smallmatrix}$, cf \cite{Oh2}, and they act on the basis via
$$\begin{array}{c c c c c c c} v_{2}^4 &\stackrel{z}{\rightarrow} &w_2^4 & \stackrel{z}{\rightarrow} & v_1^4 & \stackrel{z}{\rightarrow} & w_1^4 \\ 
 {}_y \downarrow & & {}_y \downarrow & & {}_y \downarrow \\
w_2^3 & \stackrel{z}{\rightarrow} & v_1^3 & \stackrel{z}{\rightarrow} & w_1^3\\
 {}_y \downarrow & &{}_y \downarrow\\
v_1^2 & \stackrel{z}{\rightarrow} & w_1^2\\
{}_y \downarrow\\
w_1^1
\end{array}.$$

In the general case, it is then straightforward to compute the simultaneous centralizer of $z$ and $y$. We give the main steps of this computation below.
Write $W'=\langle w'_j\mid j\in [\![1,r]\!]\rangle\subset V_a$; 
$W_1=\langle w_1^i\mid i\in  [\![1, 2l+1+\epsilon]\!] \rangle\subset V_a$, and 
$V_{init}=\K v^{2l+2}_{l+1}\subset V_b$  (resp. $\K w^{2l+1}_{l+1}\subset V_a$) if $\epsilon=1$ (resp. $\epsilon=0$). 
One finds that
$$\gl(V)^{z,y}=\langle z^iy^j\mid i\in[\![0, 2l+1+\epsilon]\!], j\in[\![0, 2l+1+\epsilon-i]\!]\rangle \oplus  \Hom(V_{init},W')\oplus\Hom(W',W_1)\oplus \gl(W').$$
Observe that $z^iy^j\in \kk$ (resp. $\in\pp$) if $i+j$ is non-zero and even (resp. odd) while $\Hom(V_{init},W')\subset \kk$  (resp. $\subset\pp$) if $\epsilon=0$ (resp. $\epsilon=1$). 
Moreover $\Hom(W',W_1)\oplus \gl(W_1)\subset \gl(V_a)$.
With respect to the above described decomposition, this provides 
\begin{eqnarray*}\dim \kk^{z,y}&=&((l+1)^2+((1-\epsilon)r)+((2l+1+\epsilon)r)+(r^2))-1\\
&=&(l+1+r)^2-1
\end{eqnarray*}
On the other hand, from \eqref{miniLevi}, one deduces 
\begin{eqnarray}\label{overp}
\dim \mf&=&(n_a^2+n_b^2-1)-(2n_a n_b)+(n_b)\notag\\
&=&(n_a-n_b)^2 -1 +n_b\notag\\
&=&((l+1+r)^2-1)+ n_b\notag\\
&=&\dim \kk^{z,y}+\rk_{sym}(\g,\kk).
\end{eqnarray}
This provides a non trivial rigid pair.

The next three examples require more complicated computations. We will only give the pair $(z,y)$ without detailed computation.

Let $V$ be a vector space of dimension $12$ and $(v_1, \dots,v_{12})$ be a basis of $V$.
Let $T,J\in \gl(V)$ be defined by $T.v_i=(-1)^i v_{13-i}$ and $J.v_i=\left\{\begin{array}{l l}(-1)^{i+1}v_i& \mbox{ if } i\in[\![1,6]\!]\\
(-1)^{i}v_i& \mbox{ if } i\in[\![7,12]\!] \end{array}\right.$. 
Let $\g:=\{x\in \gl(V)\mid {}^txT+Tx=0\}\cong \spn_{12}$ be the set of elements preserving the skew-symmetric bilinear form $\omega(x,y)={}^txTy$ and 
let $\kk:=\{x\in \g \mid x=JxJ^{-1}\}$. Since ${}^tJTJ=T$, $J^2=Id$ and the eigenspaces of $J$ are of dimension $6$, 
one deduces from \cite[Theorem 3.4]{GW} that $(\g,\kk)\cong (\spn_{12},\spn_6\oplus\spn_6)$ is of {\bf type CII}.
Set $V_a=\langle v_i\mid i\in\{1,3,5,8,10,12\}\rangle$ (resp. $V_b=\langle v_i\mid i\in\{2,4,6,7,9,11\}\rangle$) to be the $+1$(resp.$-1$)-eigenspace of $J$. 
Then, the $ab$-diagram of the elements we are interested in is $\begin{smallmatrix} a & b & a \\ a & b & a\\ a & b\\ b & a\\ b\\ b\end{smallmatrix}$, cf \cite{Oh2}.
We define $z$ via $z.v_i=\left\{\begin{array}{l l}v_{i-1}& \mbox{ for } i\in\{2,3,5,9,11,12\}\\ 0 & \mbox{ for } i\in\{1,4,6,7,8,10\} \end{array}\right.$. 
In a similar way, we define $y$ by $y.v_{2}=y.v_{7}=v_8$, $y.v_{3}=v_9$, $y.v_{4}=-v_{10}$, $y.v_{5}=v_6-v_{11}$, $y.v_{6}=v_1$, $y.v_{12}=v_7$ and finally 
$y.v_i=0$ for $i\in\{1,8,9,10,11\}$.
The following facts are easy to check:
\begin{itemize}\item $z, y\in \pp$ and they have the above given $ab$-diagram,
\item $[z,y]=0$, 
\item $\dim \kk^{z,y}=6=\dim \mf-3$ and $i((\g,\kk),(z,y))=-3$.
\end{itemize}

We construct a similar pair in {\bf type DIII} as follows. Keep $V$ of dimension $12$.
Define $T,J\in \gl(V)$ by $$T.v_i=\left\{\begin{array}{l l}(-1)^{i}v_{13-i}& \mbox{ if } i\in[\![1,6]\!]\\
(-1)^{i+1}v_{13-i}& \mbox{ if } i\in[\![7,12]\!]\end{array}\right., \qquad 
J.v_i=\sqrt{-1}\times(-1)^{i+1} v_i.$$
Then $(\g,\kk)$ can be defined as in the previous case and $(\g,\kk)\cong(\so_{12},\gl_6)$.
Set $V_a=\langle v_i\mid i\mbox{ odd }\rangle$ (resp. $V_b=\langle v_i\mid i\mbox{ even }\rangle$) to be the $+\sqrt{-1}$(resp. $-\sqrt{-1}$)-eigenspace of $J$.
We are interested in elements with $ab$-diagram $\begin{smallmatrix} a & b & a \\ b & a & b\\ a & b\\ a & b\\ a\\ b\end{smallmatrix}$, cf. \cite{Oh2}.
Choose $z$ (resp. $y$) such that it acts on the basis $v_i$ horizontally (resp. vertically) in the following way:
 $$\xymatrix{
 \ar[r]^z \ar[d]_y &  &&  v_{12} \ar[r] \ar[d]& v_{11} \ar[r] \ar[d]& v_{10} \ar[r] \ar[d]& 0          &               & v_7 \ar[r]  \ar[d]& 0  \\
  &  & & -v_{5} \ar[r]  \ar[d]& -v_4 \ar[r]      \ar[d]&  0                &             & v_9 \ar[r] \ar[d]& v_8 \ar[r] \ar[d]& 0 \\
  &  &  &-v_6 \ar[r] \ar[d] & 0                    &                 &  v_3 \ar[r]\ar[d]&v_2 \ar[r]\ar[d]& v_1 \ar[r]\ar[d]&0 \\
  &  & & 0               &                       &                &  0            &      0       & 0&
}$$
Again, it is easy to see that: 
\begin{itemize}\item $z, y\in \pp$ and they have the above given $ab$-diagram,
\item $[z,y]=0$, 
\item $\dim \kk^{z,y}=6=\dim \mf-3$ and $i((\g,\kk),(z,y))=-3$.
\end{itemize}

The last pair we are interested in is related to the orbit $\Od_{11}$ (cf. \cite{Dj1}) of the symmetric Lie algebra of {\bf type EVII}. 
We proceed as in proof of Proposition $\ref{miracle}$.
If $z\in\Od_{11}$, computations show that (cf. also \cite{JN}) $\pp(z,0)=\{0\}$, $\dim \kk(z,0)=16$, $\dim \pp(z,1)=8=\dim \kk(z,1)$, 
$\dim \pp(z,2)=4$, $\dim \kk(z,2)=12$, $\dim \pp(z,3)=4$ and $\pp(z,i)=\{0\}$ for $i\geqslant 4$.
Moreover, the $\kk(z,0)$-module $\pp(z,1)$ is isomorphic to the $\sln_4\oplus\tf_1$-module $\K^4\oplus\K^4$ 
so $$[\kk(z,0);y]=\pp(z,1)$$ for some generic element $y\in \pp(z,1)$. 
One can also find some $y\in \pp(z,1)$ such that $[\kk(z,i);y]=\pp(z,i+1)$ for $i=1,2$.
Since these are open conditions on $\pp(z,1)$, there exists $y\in\pp(z,1)$ such that $[\kk^z;y]=\pp^z$.
Then \eqref{miniLevi} gives for such $y$
$$\dim \kk^{z,y}=\dim \kk^z-\dim \pp^z=\dim \mf-\rk_{sym}(\g,\kk); \qquad i((\g,\kk),(z,y))=-3.$$

Combining this with Proposition \ref{miracle}, we have therefore shown the following proposition.
Recall that the type of each symmetric Lie algebra is given in table~\ref{table1}.
\begin{prop}\label{rigidpairsref}
There exists some non-trivial rigid pairs in the following symmetric Lie algebras:
$$(\sln_{p+1},\sln_p\oplus\tf_1) \mbox{ AIII } (p\geqslant 2), \quad (\spn_{2p+2}, \spn_{2p}\oplus\spn_2) \mbox{ CII } (p\geqslant 2), \quad (\spn_{12},\spn_6\oplus\spn_6) \mbox{ CII }(p=q=3),$$
$$(\so_{12},\gl_6) \mbox{ DIII } (n=6), \quad (\mathfrak e_7, \mathfrak e_6\oplus\tf_1) \mbox{ EVII }, \quad (\mathfrak f_4,\spn_6\oplus\sln_2)\mbox{ FII }.$$
\end{prop}
\begin{remarks}
(i) In the previous proposition, we have not pointed out rigid pairs of other symmetric Lie algebras of type AIII since we will not need it further. 
The author has also found some rigid pairs in $(\so_{10},\gl_5)$ of type DIII ($n=5$).\\ 
(ii) The two pairs $(x,y)\in\CC(\g,\kk)$ in type EVI and EIX described in \cite[Theorem 5.1]{PY} and which satisfy $\dim \g^{x,y}\leqslant \dim \mf+\dim \af$ are also rigid.
\end{remarks}

\subsection{Estimates of the integers $d_t(\g,\kk)$} 
\label{resultdt}
The following table is the core of our results on $\codim_{\CC_0} \CC_0^{irr}$. The first, second and third column define a symmetric pair $(\g,\kk)$. 
The fourth (resp. fifth) column gives a lower (resp. upper) bound for $d_t(\g,\kk)$. The bounds do not depend on $t\in\{0,1\}$. 
In a significant number of cases, these bounds are the same, and are therefore equal to $\codim_{\CC_0} \CC_0^{irr}$ thanks to corollary \ref{codimc0JKcor}.
Finally, the last column provides an example of $R\in R(\g,\kk)$ such that $c_t(R)$ gives the upper bound for $d_t(\g,\kk)$ of column five.

\hspace{-1cm}\begin{longtable}{|c|c|c|c|c|c|}
\caption{\label{table2} Bounds for $d_t(\g,\kk)$}\\
\hline
Type & $(\g,\kk)$ & assumptions & $\leqslant d_t(\g,\kk)$& $d_t(\g,\kk)\leqslant $ & $((\g',\kk'),K.z)$\\
\hline
\endhead
\hline
AI & $(\sln_n,\so_n)$ & / & 3 & 3 & $((\sld,\so_2),0)$\\
\hline
AII & $(\sln_{2n},\spn_{2n})$ & / & 6 &6 &$ ((\sln_4,\spn_4),0)$\\
\hline

\multirow{5}{*}{AIII} & \multirow{5}{*}{\begin{tabular}{c}$(\sln_{p+q},\sln_p\oplus\sln_q\oplus\tf_1)$\\ $p\leqslant q$\end{tabular}} & $p=q\geqslant 1$ & 3& 3& $((\sld,\so_2),0)$\\
\cline{3-6}
& & $1\leqslant p=q-1$ & 2 & 2 & $\left((\sln_3,\sln_2\oplus\tf_1), \begin{smallmatrix} a & b \\ a& \end{smallmatrix}\right)$\\
\cline{3-6} 
& & $1\leqslant p=q-2$ & 3 & 3 & $\left((\sln_4,\sln_3\oplus\tf_1), \begin{smallmatrix} a & b \\ a& \\ a&\end{smallmatrix}\right)$\\
\cline{3-6}
& & $1=p\lnq q$ & $q$ & $q$ & $\left((\g,\kk), \begin{smallmatrix} a & b \\ \vdots & \\a& \end{smallmatrix}\right)$\\
\cline{3-6}
& & $ 2 \leqslant p\leqslant q-3$ & 3& 4 & $(\sld\oplus\sld,\sld)$\\
\hline

\multirow{2}{*}{BDI} & \multirow{2}{*}{\begin{tabular}{c}$(\so_{p+q},\so_p\oplus\so_q)$\\ $p\leqslant q$\end{tabular}} & $1=p\lnq q$ & $q+1$ & $q+1$ & $((\g,\kk),0)$\\
\cline{3-6}
& & $2\leqslant p$ & 3 & 3 & $((\sld,\so_2),0)$\\
\hline

CI & $(\spn_{2n}, \gl_{n})$ & / & 3 & 3 & $((\sld,\so_2),0)$\\
\hline

\multirow{6}{*}{CII} & \multirow{6}{*}{\begin{tabular}{c}$(\spn_{2p+2q},\spn_{2p}\oplus\spn_{2q})$\\ $p\leqslant q$\end{tabular}} & $p=q=1$ & 5 & 5& 
$((\spn_{4},\spn_{2}\oplus\spn_{2}),0)$\\
\cline{3-6}
& & $2 \leqslant p=q$ & 3 & 4 & $((\spn_{8},\spn_{4}\oplus\spn_{4}), \begin{smallmatrix} a&b \\ b&a \\a & b\\  b &a \end{smallmatrix})$ \\
\cline{3-6}
& & $1=p\lnq q$ & $2q-1$ & $2q-1$ & \multirow{3}{*}{$\left((\spn_{2q-2p+4},\spn_{2}\oplus\spn_{2q-2p+2}),
\begin{smallmatrix} a&b \\ b&a \\a & \\  \vdots &\\ a & \end{smallmatrix}\right)$}\\
\cline{3-5}
& & $2\leqslant p=q-1$ & 2 & 3 & \\
\cline{3-5}
& & $2\leqslant p=q-2$ & 2 & 5 & \\
\cline{3-6}
& & $2\leqslant p\leqslant q-3$ & 2 & 6 & $((\sln_4,\spn_4),0)$\\
\hline

\multirow{2}{*}{DIII} & \multirow{2}{*}{\begin{tabular}{c}$(\so_{2n},\gl_n)$\\ $n\geqslant 4$ \end{tabular}} & $n$ odd & \multirow{2}{*}{3}&\multirow{2}{*}{3} & 
$\left((\sln_4,\sln_3\oplus\tf_1), \begin{smallmatrix} a & b \\ a& \\ a&\end{smallmatrix}\right)$\\
\cline{3-3} \cline{6-6} 
& & $n$ even & & & $((\sld,\so_2),0)$\\
\hline

EI & $(\mathfrak e_6, \spn_8)$& / &3 &3 & $((\sld,\so_2),0)$\\
\hline
EII & $(\mathfrak e_6, \sln_6\oplus\sln_2)$ &/ & 3 & 3 & $((\sld,\so_2),0)$\\
\hline
EIII & $(\mathfrak e_6, \so_{10}\oplus\tf_1)$ & / & 2 & 5 & $((\sln_6,\sln_5\oplus\tf_1),\begin{smallmatrix} a & b \\ \vdots & \\a& \end{smallmatrix})$\\
\hline
EIV & $(\mathfrak e_6, \mathfrak f_4)$ &/ & 3 & 10 & $((\so_{10},\so_9),0)$ \\
\hline
EV & $(\mathfrak e_7, \sln_9)$& / & 3 & 3 & $((\sld,\so_2),0)$\\
\hline
EVI & $(\mathfrak e_7, \so_{12}\oplus\sln_2)$& / & 3 & 3 & $((\sld,\so_2),0)$\\
\hline
EVII & $(\mathfrak e_7, \mathfrak e_6\oplus\tf_1)$& / & 3 & 3 & $((\sld,\so_2),0)$\\
\hline
EVIII & $(\mathfrak e_8, \so_{16})$&/ & 3 & 3 & $((\sld,\so_2),0)$\\
\hline
EIX & $(\mathfrak e_8, \mathfrak e_7\oplus\sln_2)$&/ & 3 & 3 & $((\sld,\so_2),0)$\\
\hline
FI & $(\mathfrak f_4, \spn_6\oplus\sln_2)$&/ & 3 & 3 & $((\sld,\so_2),0)$\\
\hline
FII & $(\mathfrak f_4, \so_9)$& /& 5& 5&$((\g,\kk),\mathcal O_1)$\\
\hline
GI & $(\mathfrak g_2, \sln_2\oplus\sln_2)$& /& 2 & 2 &$((\g,\kk), \mathcal O_4)$\\
\hline
\end{longtable}

The main steps leading to this table are the following.
We have already seen in sections \ref{rankonecase} and \ref{specialcase} that the last column gives an example of $R\in R_1(\g,\kk)$ such that 
$c_t(R)$ has the value given in the fifth column. 
Our numbers do not depend on $t\in\{0,1\}$, since the cases have either been treated in section \ref{rankonecase} or satisfy $\CC=\CC_0$. 
Since $d_t(\g,\kk)$ is a minimum \eqref{defdt}, we get the informations of the fifth column.

It remains to prove that $d_t(\g,\kk)$ is greater than the integer given in the fourth column. 
This relies on a case by case computation of $c_t(R)$ for each $R\in R_1(\g,\kk)\subset R(\g,\kk)$, cf. Lemma \ref{trick1}. 
Since $c_0(R)\geqslant c_1(R)$, it is sufficient to get a lower bound for $c_1(R)$.
Recall that we may forget the computation of $c_1(R)$ in a significant number of cases, thanks to corollary \ref{pdistbis}. 
Recall also that the finite set $R(\g,\kk)$, defined in \eqref{Rgk}, is easily determined since standard reduced $\pp$-Levi $(\g',\kk')$ are given by Satake subdiagrams of $S(\g,\kk)$, 
cf. section \ref{pLevisdist}, and nilpotent orbits of symmetric Lie algebras have been classified in \cite{Oh1,Oh2,Dj1,Dj2}. 
The remaining of the section is devoted to explain the necessary computations.\\

Section \ref{rankonecase} gives $d_t(\g,\kk)$ in cases {\bf AIII} ($p=1$), {\bf BDI} ($p=1$), {\bf CII} ($p=1$) and {\bf FII}.
Note that $d_t(\g,\kk))\geqslant 2$ for all $(\g,\kk)$ of symmetric rank one. 
In the general case,
since $c_t(L,K.z)\geqslant \rk_{sym}(L)$ for each standard reduced $\pp$-Levi $L\in L(\g,\kk)$, we have $d_t(\g,\kk)\geqslant 2$ for all symmetric Lie algebra $(\g,\kk)$.
This provides $d_t=2$ in the two cases {\bf AIII} ($p=q-1$) and {\bf GI}. This also gives our best bound for {\bf CII}  ($2\leqslant p\leqslant q-1$) and {\bf EIII} . \\

One  also has to examine the case of simple symmetric Lie algebras $\g'=\kk'\oplus\pp'$ of symmetric rank $2$ 
whose reduced $\pp'$-Levi $L$ of rank one satisfy $d_t(L)\geqslant 3$ 
(\emph{i.e.} $L\not \cong (\sln_{2p+1},\sln_{p}\oplus\sln_{p+1}\oplus \tf_1)$).

Assume first that $(\g',\kk')$ has a reduced $\pp'$-Levi $L$ of rank one satisfying $d_t(L)=3$. 
In this case, proving $d_1(\g',\kk')=3$ (and hence $d_0(\g',\kk')=3$) is equivalent to show that 
for each non-regular $\pp'$-distinguished element $z\in\pp'$, there exists $y\in(\pp')^z$ such that $i((\g',\kk'),(y,z))\leqslant 0$, cf. Corollary \ref{pdistbis}. 
It appears that such an element $y$ has always been found when sought, except for the orbit $\mathcal \Od_4$ in the GI case, as  already noticed in section \ref{specialcase}. 
Because of the large number of orbits to consider, computations has not been made for $(\spn_{2q+4}, \spn_{2q}\oplus\spn_4)$ ($q\geqslant 3$) of type CII and $(\mathfrak e_6, \so_{10}\oplus\tf_1)$ of type EIII.

For the remaining such $(\g',\kk')$, we list a characteristic object related to each of their $\pp'$-distinguished orbit, 
as explained in the beginning of section \ref{specialcase}. 
The pair $(z,y)$ will not be explicitly given in general, since this is not very enlightening.
The $\pp'$-distinguished orbits have already been classified in  \cite{PT} and \cite{Bu1}. 
We refer to \cite{Bu1}, since the characteristic objects considered in the present work are the same.
As a by-product, this will also show that $d_t(\g,\kk)=3$ for all pair $(\g,\kk)$ (different from $(\sln_{2p+1}, \sln_p\oplus\sln_{p+1}\oplus\tf_1)$) 
having only such $(\g',\kk')$ as $\pp$-Levi of rank two. 
Each time that such a result is obtained, the type of $(\g,\kk)$ will be marked in bolds characters.\\

Concerning $(\g',\kk')=(\sln_3,\so_3)$, the only $\pp'$-distinguished orbit is the regular one, hence Corollary \ref{pdistbis} together with our upper bound give $d_t(\sln_3,\so_3)=3$.
This solves the cases {\bf AI}, {\bf EI}, {\bf EV} and {\bf EVIII}.

Consider $(\g',\kk')=(\so_{q+2},\so_p\oplus\so_2)$ with $q\geqslant 2$. 
This symmetric algebra is simple when $q\geqslant 3$.
If $q=3$ there is a single non regular nilpotent orbit $K.z$ which is $\pp'$-distinguished. 
It corresponds to the $ab$-diagram $\begin{smallmatrix}b&a&b\\a \\ a\end{smallmatrix}$. 
In the same way, if $q\geqslant 4$, there are only two $ab$-diagrams corresponding to $\pp'$-distinguished orbits. 
They are $\begin{smallmatrix}b&a&b\\a \\\vdots \\ a\end{smallmatrix}$ and $\begin{smallmatrix}a&b&a \\a&b&a \\a \\\vdots \\ a\end{smallmatrix}$.
For each of these orbits, the author has found some $y\in(\pp')^z$ such that $(y,z)$ is a principal pair of $(\g',\kk')$. 
This shows that $d_t(\g,\kk)=3$ for the pairs of the following types whose $\pp$-Levi of rank two are of type AI or BDI :  {\bf BDI} ($p\geqslant 2$), 
{\bf CI}, and {\bf FI}.

Now, set $(\g',\kk')=(\mathfrak e_7, \sln_9)$ of type EIV. The only $\pp'$-distinguished orbit is the regular one, thus we can say that $d_t(\g',\kk')\geqslant 3$. 
The symmetric pairs  $(\g,\kk)$ of type {\bf EIV}, {\bf EVII} and {\bf EIX} have $\pp$-Levi of rank two which are all of type AI, BDI or EIV. 
Hence $d_t(\g,\kk)\geqslant 3$ for these pairs.

If $(\g',\kk')=(\so_{2n},\gl_n)$, of type DIII, is of symmetric rank $2$, one has $n=4$ or $n=5$. If $n=4$, $(\g',\kk')$ is isomorphic to $(\so_8,\so_6\oplus\so_2)$, of type DI. 
So we can assume that $n=5$ and we have to consider
$$\begin{smallmatrix} a& b& a &b\\a&b&a&b\\a\\b\end{smallmatrix}, \qquad
\begin{smallmatrix} b& a& b &a\\b&a&b&a\\a\\b\end{smallmatrix}, \qquad \begin{smallmatrix} a& b &a\\b&a&b\\b & a\\b& a\end{smallmatrix}, \qquad
\begin{smallmatrix} a& b &a\\b&a&b\\a & b\\a& b\end{smallmatrix}, \qquad \begin{smallmatrix} a& b \\a& b \\a& b \\a& b\\a \\b \end{smallmatrix}, \qquad 
\begin{smallmatrix} b& a \\b& a \\b& a \\b& a\\ a\\b \end{smallmatrix}.$$
These orbits are involved in semi-rigid pairs. 
This does not provide new bounds for $d_t(\g,\kk)$ for new pairs $(\g,\kk)$ of rank greater than $2$.\\

Next, consider $(\g',\kk')=(\sln_{p+q}, \sln_p\oplus \sln_q)$ with $2=p\leqslant q$.
The rank one $\pp'$-Levi are isomorphic to $(\sld\oplus\sld,\sld)$ and $(\sln_q,\sln_{q-1}\oplus\tf_1)$.
If $q=2$ (or $p=q$), there are two non-regular $\pp'$-distinguished $K$-orbits of respective $ab$-diagram 
$\begin{smallmatrix} a&b\\ a & b\end{smallmatrix}$ and  $\begin{smallmatrix} b&a\\ b & a\end{smallmatrix}$, which are involved in principal pairs.
This solves cases {\bf EII} and {\bf AIII} ($p=q$).
Case $q=3$ (or $p=q-1$) having already been settled,  we assume that $q\geqslant 4$ (or $p\leqslant q-2$). 
The non-regular $\pp'$-distinguished orbits have the following $ab$-diagrams 
$$\begin{smallmatrix} a&b & a &b \\ a \\\vdots \\ a\end{smallmatrix},\qquad \begin{smallmatrix} b&a & b &a \\ a \\\vdots\\ a\end{smallmatrix}, 
\qquad \begin{smallmatrix} b&a & b  \\ a \\ \vdots\\a\end{smallmatrix}, \qquad \begin{smallmatrix} a&b & a  \\a&  b &a \\ a \\ \vdots\\a\end{smallmatrix}, 
\qquad \begin{smallmatrix} a&b  &a\\a&  b\\ a  \\ \vdots\\a\end{smallmatrix}, \qquad \begin{smallmatrix} a&b&a  \\b&  a\\ a  \\ \vdots\\a\end{smallmatrix}
\qquad \begin{smallmatrix} a&b  \\a&  b\\ a  \\ \vdots\\a\end{smallmatrix}, \qquad \begin{smallmatrix} b&a  \\b&  a\\ a  \\ \vdots\\a\end{smallmatrix}.$$
Again, one finds pairs such that $i((\g',\kk'),(z,y))\leqslant 0$ for each orbit and this gives our best bound for {\bf AIII} ($p\leqslant q-3$). 
If $q\geqslant 5$, the table given at the end of section \ref{rankonecase} shows that $d_t(\sld\oplus\sld,\sld)=4$, while we have already seen that $d_t(\sln_q,\sln_{q-1}\oplus\tf_1)=q-1\geqslant 4$.
One can also show that $\codim_{(\pp')^z} \II_1((\g',\kk'),z)\geqslant 2$ for each $K'$-orbit $K'.z$ of $(\g',\kk')$ thanks to Corollary \ref{pdistbis}. 
A study of the rank $3$ case is required to determine whether $d_t(\g',\kk')$ is equal to $3$ or $4$ when $p\geqslant 3$.
The study of eleven classes of $ab$-diagrams of $\pp'$-distinguished orbits would be necessary.
\\

In the case $(\g',\kk')=(\spn_{8},\spn_{4},\oplus\spn_{4})$, of type CII ($p=q$), there are four non-regular orbits. Their $ab$-diagrams are 
$$\begin{smallmatrix} a& b& a \\a&b&a \\b\\b\end{smallmatrix}, \qquad \begin{smallmatrix} b& a& b \\b&a&b \\a\\a\end{smallmatrix}, \qquad
\begin{smallmatrix} a& b \\b&a\\a& b\\b& a\end{smallmatrix}, \qquad \begin{smallmatrix} a& b \\b&a\\a\\ a\\b\\b\end{smallmatrix}.$$
Some computations, similar to the first one of section~\ref{specialcase}, show that $\codim_{(\pp')^z}\II_t((\g',\kk'),z)=4$ for any element $z$ belonging to each of the two first orbits.
The third one has already been considered in section \ref{specialcase}.
It is easy to see that any element $z\in\pp$ with the last $ab$-diagram commutes with a regular nilpotent element, satisfies $\delta(z)=1$ and $\NN\cap(\pp')^z$ is irreducible \cite[\S3.3]{Bu1}, hence $\codim_{(\pp')^z}\II_t((\g',\kk'),z)\geqslant 2$ thanks to corollary \ref{pdistbis}.
We will be able to conclude in type CII ($p=q$) when the AII case will be checked.
In order to know whether $d_t(\g,\kk)$ is equal to $3$ or $4$, we would have to consider the rank $3$ case. 
In $(\spn_{12},\spn_{6},\oplus\spn_{6})$, there is one $\pp'$-distinguished orbit which gives rise to a rigid pair, cf. \ref{specialcase}, 
and three other $\pp'$-distinguished orbit to consider.\\

We now consider the case $(\g',\kk')=(\sln_{2n},\spn_{2n})$, of type AII. 
We need to deal with rank up to $5$ (\emph{i.e.} $n=6$ and $(\sln_{12},\spn_{12})$) since the upper bound for $(\g,\kk)$ is equal to $6$. 
The lemmas of section \ref{shortcut} are useful to reduce the number of cases to check. 
It follows from the classifications of \cite{PT,Bu1} that the only $\pp'$-distinguished orbit is the regular one. 
Hence, Corollary \ref{pdistbis} shows that it is enough to prove that $\codim_{(\pp')^z} \II_t((\sln_{2n},\spn_{2n}),z)\geqslant 7-n$ for each $n\in [\![3,5]\!]$ and each $K'$-orbit $K'.z$ of $(\g',\kk')$.
In rank $2$ ($n=3$), we have two orbits to consider, the zero one which is easily settled and the orbit corresponding to the doubled partition $(2^2,1^2)$.
For this last orbit, one finds $\codim_{(\pp')^z} \II_t((\sln_{2n},\spn_{2n}),z)=5$. 
This gives $d_t(\g,\kk)\geqslant 3$ for pairs {\bf CII} ($p=q$) {\bf EVI} and  {\bf DIII}\\
In rank $3$ ($n=4$), the non regular orbits have the following doubled partitions $$(3^2,1^2), \qquad (2^2,2^2),\qquad (2^2,1^2,1^2), \qquad (1^2,1^2,1^2,1^2)$$
which are checked case by case.
The last remaining case is $(\sln_{10},\spn_5)$, of rank $4$ ($n=5$), where we have to show that $\codim_{(\pp')^z} \II_t((\sln_{2n},\spn_{2n}),z)\geqslant 2$.
For the orbits whose doubled partition has at least three parts, the result follows from corollary \ref{pdistbis} since their defect is at least two (cf. \cite{Bu1}).
It remains the case of orbits corresponding to $(3^2,2^2)$ and $(4^2,1^2)$. If $z$ has doubled partition $(3^2,2^2)$, we can note that $\delta(z)=1$, $\NN\cap(\pp')^z$ is irreducible \cite{Bu1} and that $z$ commutes with a regular nilpotent element
so $\II_t((\sln_{2n},\spn_{2n}),z)\subset(\pp')^z\cap\NN\cap(\pp')^{irr}$ is of codimension at least two in $(\pp')^z$ thanks to corollary \ref{pdistbis}. 
The other case can also be checked and one gets the result for {\bf AII}.

\section{Geometric consequences for the commuting variety}
\label{overprincipal}
\subsection{Rigid pairs}
In the previous sections we have met several examples of rigid pairs (cf Definition \ref{ovpr}). In this section, we investigate properties of these pairs. In particular, we will establish a connection with the nilpotent pairs, which are sometimes called \emph{nilpairs}.
\begin{defi} \label{rigid} A pair of commuting elements $(x,y)\in\CC(\g,\kk)$ is said to be 
\emph{nilpotent} if for all $t_1, t_2\in \K^*$, there exists $k\in K$ such that $(k.x,k.y)=(t_1x,t_2y)$. 
\end{defi}
V.~Ginzburg \cite{Gi} has obtained important results for principal nilpotent pairs in semisimple Lie algebras.
Combining Definitions \ref{ovpr} and \ref{rigid}, we get an equivalent definition to Ginzburg's principal nilpotent pairs in type 0. 
It is shown in \cite{Gi} that there is a finite number of orbits of principal nilpotent pairs, but it exists some infinite families of nilpotent pairs. 
It is also shown that the notion of characteristic attached to a nilpotent element has a doubled analogue for nilpotent pairs.
Moreover, Ginzburg pointed out several links between the centralizer of the nilpotent pair and the centralizer of its characteristic. 
Shortly after \cite{Gi}, Elashvili and Panyushev gave a classification of all principal nilpotent pairs of semisimple Lie algebras in \cite{EP1, EP2}.
One can find in \cite{Gi, Pa00, Pa01, Yu} the definition of distinguished, almost principal nilpotent, even, almost even, excellent or wonderful pairs. These notions aim to fit in the space between nilpotent pairs and principal nilpotent pairs.

However, in the symmetric Lie algebra setting, rigid pairs do not satisfy natural analogues of most of these notions. 
In the following, we only list some basic remarks about rigid pairs.
The first elementary result is the following.
\begin{lm}\label{opdense}
Let $(x,y)$ be a commuting pair.\\
(i) One has $\dim \kk^{x,y}\geqslant \dim \kk-\dim \pp$.\\
(ii) $(x,y)$ is rigid if and only if $K^x.y$ is dense in $\pp^x$.
\end{lm}
\begin{proof}
Consider the $K^x$-orbit of $y$. It lies in $\pp^x$, therefore
$$\dim K^{x,y}= \dim \kk^x-\dim \pp^x+\codim_{\pp^x} K^x.y=\dim \kk-\dim \pp+\codim_{\pp^x} K^x.y.$$
Since $\kk^{x,y}$ is the Lie algebra of $K^{x,y}$, we get the results from \eqref{miniLevi}.
\end{proof}
\begin{remark}\label{rksop}
As a consequence, the rigid pairs have a minimal irregularity number, equal to $-\rk_{sym}(\g,\kk)$.
It follows from lemma \ref{igk}(i) that for any commuting pair $(x,y)$, one has $$i((\g,\kk),(x,y))=i((\g_{x_s}, \kk_{x_s}), (x_s,y_2))\geqslant -\rk_{sym}(\g_{x_s},\kk_{x_s}).$$
From the upper semi-continuity of the map $(x,y)\mapsto i((\g,\kk),(x,y))$, it then follows that $K.x$ and $K.y$ are rigid orbits in $(\g,\kk)$ if $(x,y)$ is a rigid pair.
\end{remark}

\begin{lm}\label{nilpair}
A rigid pair is nilpotent and the set of $K$-orbits of rigid pairs of $(\g,\kk)$ is finite.
\end{lm}
\begin{proof}
Let $(x,y)$ be a rigid pair of $(\g,\kk)$ and let $t_1,t_2\in\K^*$. 
From lemma \ref{opdense} we know that the orbit $K^x.y$ is a dense open subset of $\pp^x$. 
This implies that the set of $y'\in \pp^x$ such that $(x,y')$ is rigid is a single $K^x$-orbit.
In particular, there is at most one orbit of rigid pairs attached to each nilpotent $K$-orbit. This yield the finiteness assertion.
Since $\kk^{x,y}=\kk^{x,t_2y}$, the pair $(x,t_2y)$ is rigid for any $t_2\in\K^*$, hence it belongs to $K.(x,y)$.
By symmetry, we can then apply the same argument to show that $(t_1x,t_2y)$ belongs to the $(K^{t_2y}=)K^y$-orbit
of $(x,t_2y)$.
\end{proof}
\begin{remark}
In particular, if $(x,y)$ is rigid, then $x$ and $y$ are nilpotent. Observe that $K.y\cap\pp^x$ is a dense open subset of $\pp^x$, cf. Lemma \ref{opdense}(ii),
hence $\pp^x\subset\mathcal N$ and $x$ is $\pp$-distinguished. 
Now, recall that $\pp$-distinguished elements are $\pp$-self large, \emph{i.e.} $\pp^x\subset \overline{K.x}$ \cite[38.10.6]{TY}, therefore
$x$ and $y$ are in the same $K$-orbit.
\end{remark}

Following \cite{Gi} one can also find pairs of semisimple elements associated to rigid pairs.
To do this, we need \cite[Lemma 1.3]{Gi} whose statement and proof can be easily translated in the symmetric case. 
\begin{lm}
For any nilpotent commuting pair $(e_1,e_2)$, there exists $m_1, m_2\gnq0$ and an algebraic group homomorphism $\gamma:\K^*\times\K^*\rightarrow K$ such that 
$$ \gamma(t_1,t_2) (e_i)=t^m_i e_i, \; \forall (t_1,t_2)\in \K^*\times\K^*, \; i=1,2.$$
\end{lm}
Consider now $\gamma_1:t_1\rightarrow\gamma(t_1,1)$ and $\gamma_2:t_2\rightarrow\gamma(1,t_2)$. 
Then, the Lie algebra homomorphism $(d\gamma_i)_{\mid t_i=1}: \K\rightarrow\kk$ maps $\frac 1{m_i}$ to an element $h_i$ satisfying $[h_i,e_j]=\delta_{i,j} e_i$
where $\delta_{i,j}$ is Kronecker's delta.
It is not obvious from Ginzburg's methods that $\ad h_i$ has integral eigenvalues. Actually, the proof of \cite[Proposition~1.9]{Gi} consists in deforming $\g^{(h_1,h_2)}$ in  $\g^{(e_1,e_2)}$. Since it may happen that $\dim \wfr^{(h_1,h_2)}>\dim \wfr^{(e_1,e_2)}$ with $\wfr=\g$, $\kk$ or $\pp$ in the symmetric case, the method can not be applied.

\subsection{About irreducible components of the commuting variety}
\label{reducibility}
The aim of this section is to emphasize the link between the reducibility of $\CC(\g,\kk)$ and rigid pairs, in view of conjecture \eqref{conjBu} (recall the equivalent formulation given in \eqref{conjBubis}).

We begin with an easy observation.
\begin{lm}\label{red}
If there exists $((\g',\kk'),K.z)\in R(\g,\kk)$ such that $z$ belongs to a non-principal semi-rigid pair of $(\g',\kk')$, 
then there exists some non-principal semi-rigid pairs in $(\g,\kk)$ and $\CC(\g,\kk)$ is reducible.
\end{lm}
\begin{proof}
For any element $s\in\cc_{\af}(\g')^{\bullet}$, one has $\g'=\g_s$ and the element $x=s+z$ has $s$ (resp. $z$) as semisimple (resp. nilpotent) part.
Let $y\in\pp'$ such that $(z,y)$ is a non-principal semi-rigid pair of $(\g',\kk')$.
It follows from lemma \ref{igk} (i) that $i((\g,\kk),(x,y))<0$ and $(x,y)$ is a non-principal semi-rigid pair of $(\g,\kk)$.
Thus the discussion previous to \eqref{singloc1} yields to $(x,y)\notin \CC_0(\g,\kk)$.
\end{proof}

\begin{remark}Corollary \ref{cormiracle} (i) is, somehow, a converse for symmetric Lie algebras of rank one.\end{remark}

In view of lemma \ref{red}, proposition \ref{rigidpairsref} leads to the reducibility of the commuting variety in a significant number of cases. 
\begin{thm}\label{mainthm}
The symmetric commuting variety is reducible in cases AIII ($p\neq q$), CII ($p=q\geqslant 3$ or $p\neq q$), DIII ($n\notin \{1,2,4\}$), EIII, EVI, EVII, EIX and FII. 
\end{thm}
\begin{proof}
This follows from the fact that all these symmetric pair have a non-compact $\pp$-Levi isomorphic to one of the symmetric Lie algebras given in Proposition \ref{rigidpairsref}. In particular, lemma \ref{red} applies in these cases.
In order to identify a $\pp$-Levi of a given symmetric Lie algebra, we refer to \begin{itemize}
\item the correspondence  between $\pp$-Levi and Satake sub-diagrams described above Lemma \ref{correspondence},
\item the table \ref{table1} which lists the Satake diagrams given in appendix.
\end{itemize}

The reducibility of $\CC(\mathfrak f_4, \so_9)$ in type {\bf FII} follows from the existence of non-trivial rigid pairs as shown in Proposition \ref{miracle}.
The existence of a rigid pair in type AIII ($1=p<q$) proved in Proposition \ref{miracle} implies the reducibility of $\CC(\g,\kk)$ in the following cases:
\begin{itemize} 
\item $(\sln_{p+q}, \sln_p\oplus\sln_q\oplus\tf_1)$ with $p\neq q$ of type {\bf AIII},
\item $(\so_{2n},\gl_{n})$ with $n\geqslant 3$ odd, of type {\bf DIII},
\item $(\mathfrak e_6, \so_{10}\oplus\tf_1)$ of type {\bf EIII}.
\end{itemize}
In a similar way, the rigid pair in $(\spn_{2p+2},\spn_{2p}\oplus\spn_2)$ of rank one (cf. Proposition \ref{miracle}), provides the reducibility of 
$\CC(\spn_{2(p+q)}, \spn_{2p}\oplus\spn_{2q})$ of type {\bf CII} ($p\neq q$).
The rigid pair found in $(\spn_{12}, \spn_{6}\oplus\spn_{6})$ in section \ref{specialcase} gives the same result for 
$(\spn_{4p},\spn_{2p}\oplus\spn_{2p})$ with $p\geqslant 3$ which is also of type {\bf CII}.
Two more cases are solved thanks to the rigid pair of $(\so_{12}, \gl_{6})$ found in section \ref{specialcase}:
\begin{itemize}
\item $(\so_{2n},\gl_{n})$ with $n\geqslant 6$ even, of type {\bf DIII},
\item $(\mathfrak e_7, \so_{12},\oplus\sld)$ of type {\bf EVI}.
\end{itemize} 
Finally, the rigid pair obtained in EVII proves the reducibility in the last two cases:
\begin{itemize}
\item $(\mathfrak e_7, \mathfrak e_6\oplus\tf_1)$ of type {\bf EVII},
\item $(\mathfrak e_8, \mathfrak e_7\oplus\sld)$ of type {\bf EIX}.
\end{itemize} 
\end{proof}
\begin{remark}In theorem \ref{mainthm}, the only new result is given by the CII, DIII ($n$ even)and EVII cases.
\\
Recall also that the symmetric Lie algebras not occuring in theorem \ref{mainthm} have an irreducible commuting variety. 
This is proved in \cite{Pa94,Pa04,PY},  cf. table \ref{table1}. \end{remark}

The omnipresence of non-trivial rigid pairs in reducible commuting varieties originally motivated conjecture \eqref{conjBu} (cf. also \eqref{conjBubis}). 
The forthcoming Propositions \ref{component} and \ref{Po4} aim to give credit to this conjecture.
We also develop at the end of this section the example of $(\sln_6,\sln_4\oplus\sln_2\oplus\tf_1)$ in which the link between rigid pairs and irreducible components is striking.
\begin{prop}\label{component}
Let $(\g',\kk')$ be a standard reduced $\pp$-Levi of $(\g,\kk)$ and $(x,y)$ be a rigid pair of $(\g',\kk')$. Then,\\
(i) $\CC(\g',\kk')(J_{K'}(x))=\overline{K'.(x,y)}$,\\
(ii) The closed variety $\CC(J_K((\g',\kk'),K'.x))$ defined in section \ref{Jordan} is an irreducible component of $\CC(\g,\kk)$ of dimension $\dim \cc_{\af}(\g')+\dim \pp$.
\end{prop}
\begin{proof}
(i) is a straightforward consequence of Lemma \ref{opdense}(ii)\\
(ii) First, we choose a semisimple element $t\in \af$ such that  $(\g',\kk')=(\g_t,\kk_t)$.
Assume that $z\in\pp$ is such that $\CC(J_K(t+x))\subset\CC(J_K(z))$. 
Mapping these two sets through the projection on the first variable yields $\overline{J_K(t+x)}\subseteq\overline{J_K(z)}$.
It follows from lemma \ref{inclusion} that there exists $k\in K$ such that \begin{equation}\label{incl}\cc_{\pp}(\g^t) \subset k.\cc_{\pp}(\g^{z_s})\end{equation} where $z_s$ is the semisimple part of $z$. 
Replacing $z$ by $k.z$, we may assume that $k=Id$. 
We can decompose $z=z_1+ z_2\in \pp^{t}$ with respect to $\pp^{t}=\cc_{\pp}(\g^t)\oplus\pp'$, cf. \eqref{cpps}.  
Then, we get from \eqref{incl} that $z_1\in \cpg{t}$ and $z_s=z_1+(z_2)_s$.
Hence, for all $z'=z_1'+z_2'\in \pp^{z}=\cc_{\pp}(\g^t)\oplus(\pp')^{z_2}$, Remark \ref{remarkigk} and \ref{rksop} yields that
$$i((\g,\kk),(z,z'))=i((\g',\kk'),(z_2,z_2'))\geqslant -\rk_{sym}(\g_{z_s},\kk_{z_s})\geqslant \rk_{sym}(\g',\kk')=i((\g,\kk),(x,y)).$$
The upper semi-continuity of the map $i((\g,\kk),(.,.))$ implies that the previous inequalities are equalities when $z'$ runs through a dense open set of $\pp^z$.
Therefore, $z_s\in\cpg{t}$ and $\dim \CC(J_K(z))=\dim\cpg{t}+\dim\pp=\dim \CC(J_K(t+x))$. 
Since they are both closed, the two irreducible varieties $\CC(J_K(z))$ and $\CC(J_K(t+x))$ must coincide.
QED.
\end{proof}

\begin{remark}Proposition \ref{component}(ii) implies the ``if'' part of conjecture \eqref{conjBu}.\end{remark}


Now, we illustrate conjecture \eqref{conjBu} in the particular case $(\g,\kk)=(\sln_6,\sln_4\oplus\sln_2\oplus\tf_1)$. 
It follows from Proposition \ref{irred}, the classification of $\pp$-distinguished orbits of \cite{PT,Bu1} 
and the classification of rigid $K$-orbits of $\pp$ given in \cite{Bu2} that there are at most seven Jordan $K$-classes which are likely to generate an irreducible component of $\CC(\g,\kk)$.
These Jordan $K$-classes are of the form $J_K(R_i)$ ($i\in[\![1,7]\!]$) where  
$$R_1=((\mf',\mf'),\{0\}), \qquad 
R_2=((\sln_4, \sln_3\oplus\tf_1), \begin{smallmatrix} a & b \\ a \\ a\end{smallmatrix}), \qquad
R_3=((\sln_4, \sln_3\oplus\tf_1), \begin{smallmatrix} b & a \\ a \\ a\end{smallmatrix}),$$
$$ R_4=((\g, \kk), \begin{smallmatrix} a & b & a \\ b & a \\ a\end{smallmatrix}), \qquad
R_5=((\g, \kk), \begin{smallmatrix} a & b & a \\ a & b \\ a\end{smallmatrix}), \qquad
R_6=((\g, \kk), \begin{smallmatrix} a & b \\ a & b \\ a \\ a\end{smallmatrix}), \qquad
R_7=((\g, \kk), \begin{smallmatrix} b & a  \\ b & a \\ a\\a \end{smallmatrix}).$$
One easily verify that the transposition is an automorphism of $(\g,\kk)$ which sends $R_{2i}$ on $R_{2i+1}$ for $i=1,2,3$.
It follows from computations of section \ref{specialcase} that $R_2$ and $R_4$ are involved in rigid pairs and
a similar computation shows that this is also the case for $R_6$. 
Then it follows from Proposition \ref{component} that $\CC(\g,\kk)$ has exactly seven irreducible components corresponding to each $\CC(J_K(R_i))$ for $i\in[\![1,7]\!]$.
In particular, conjecture \eqref{conjBu} holds in this case.
This also shows that the lower bound given in \cite[Proposition 4.4]{PY} is not always optimal.

\subsection{Singular locus and conjecture \eqref{conjBu}}
The main motivation to study the irregular locus of $\CC(\g,\kk)$ is that it can be compared to the singular locus.
This was done in \cite[\S4]{Pa94} for symmetric Lie algebras of maximal rank. 
The presentation of results of this section is inspired by \cite[\S2]{Po}.
As noted in \cite[1.13]{Po}, a straightforward translation of his proofs gives analogous results in the symmetric setting. 
This is mostly done by replacing the $G$-action on $\g$ by the $K$-action on $\pp$ and by defining the moment map as 
$$\mu:\begin{array}{r c l}\pp\times\pp&\rightarrow &\kk\\ (x,y)&\mapsto& [x,y]\end{array}.$$
First, observe that $\CC=\mu^{-1}(0)$. Then, for any pair $(x,y)\in \CC(\g,\kk)$, one has \begin{equation}Ker\; d_{(x,y)}\mu=T_{(x,y)}(\mathfrak X(\g,\kk)).\label{TXxy}\end{equation}
The following is an analogue of \cite[Lemma 2.3]{Po}.
\begin{lm}
\label{Po1}
Let $(x,y)$ be a point of $\pp\times\pp$. Then,\\
(i) $\dim Ker\; d_{(x,y)}\mu=2 \dim \pp-\dim \kk+\dim \kk^{(x,y)}$\\
(ii) $\dim ([\pp,x]+[\pp,y])=\dim K.{(x,y)}$.
\end{lm}
Then the following lemma motivated by \cite[Lemma 2.4]{Po}.
\begin{lm}\label{Po2}
Assume that $(L,\Od)\in R(\g,\kk)$ and $(x,y)\in\CC(J_K(L,\Od))$.\\
(i) One has $\dim Ker\; d_{(x,y)}\mu\geqslant \dim \CC(J_K(L,\Od))$,\\
(ii) If $x\in J_K(L,\Od)$, then (i) is an equality for some $y\in \pp^x$ if and only if $\Od$ is involved in rigid pairs of $L$.\\
(iii) If $\Od$ is involved in rigid pairs in $L$, then the following properties are equivalent:\\
$\phantom{bla}$ (a) $\dim Ker\; d_{(x,y)}\mu>\dim  \CC(J_K(L,\Od))$,\\
$\phantom{bla}$ (b) ${(x,y)}\in \CC(J_K(L,\Od))^{irr}$.
\end{lm}
\begin{proof}
Lemma \ref{Po1}(i) and \eqref{miniLevi} give
\begin{eqnarray}\label{dimCJKLO}
\dim Ker\; d_{(x,y)}\mu&=&\dim \pp+(\dim \pp-\dim \kk+\dim \mf)+(\dim \kk^{(x,y)}-\dim \mf)\notag \\
&=& \dim \pp+\rk_{sym}(\g,\kk)+i((\g,\kk),{(x,y)})\\
&=& \dim \pp+\dim \cc_{\af}(L)+(\rk_{sym}(L)+i((\g,\kk),{(x,y)}))\notag
\end{eqnarray}
When $x$  belongs to the dense open set $J_K(L,\Od)$, claims (i) and (ii) follows from Remark \ref{rksop}, Lemma \ref{igk}(i) and Proposition \ref{component}.
In the general case, (i) follows from the upper semi-continuity of $\dim Ker\; d_{(x,y)}\mu$.
Then (iii) holds thanks to (ii), Lemma \ref{Po1}(i) and \eqref{dimCJKLO}.
\end{proof}

Recall that $(0,0)$ is a trivial rigid pair in $(\mf',\mf')$, so Lemma \ref{Po2}(iii) holds for $\CC_0$.
Furthermore, if conjecture \eqref{conjBu} is satisfied, then Lemma \ref{Po2}(iii) holds for any irreducible component of $\CC$


Statements (i) and (ii) of the next proposition are analogues of \cite[Theorem~1.3]{Po}.
\begin{prop}\label{Po4}$\phantom{a}$\\
(i) If $L\in L(\g,\kk)$ and $\Od$ is an orbit in $L$ involved in rigid pairs then $\CC(J_K(L,\Od))^{sing}\subseteq\CC(J_K(L,\Od))^{irr}$;\\
(ii) If $\CC=\CC_0$ and conjecture \eqref{conjred} holds then $\CC^{sing}=\CC^{irr}=\CC^{+}$.\\
(iii) Conjecture \eqref{conjBu} is equivalent to the following statement
\begin{equation}
\mbox{$\mathfrak X(\g,\kk)$ is generically reduced.}
\tag{$\mathcal R'$}
\label{Sprime}
\end{equation}
(iv) Conjecture \eqref{conjred} implies conjecture \eqref{conjBu}.
\end{prop}
\begin{proof}
Let $(L, \Od)$ be any element of $R(\g,\kk)$ and $(x,y)\in \CC(J_K(L,\Od))$.
Since $\CC(J_K(L,\Od))\subseteq\mu^{-1}(0)$, one has
\begin{equation}T_{(x,y)}(\CC(J_K(L,\Od)))\subseteq Ker\; d_{(x,y)}\mu=T_{(x,y)}(\mathfrak X(\g,\kk)).\label{R}\end{equation}
(i)  Assume that $(x,y)\in \CC(J_K(L,\Od))^{sing}$, then $\dim Ker\;  d_{(x,y)}\mu>\dim T_{(x,y)}(\CC(J_K(L,\Od)))>\dim \CC(J_K(L,\Od))$. Hence, ${(x,y)}\in\CC(J_K(L,\Od))^{irr}$ thanks to Lemma \ref{Po2}(iii).\\
(ii) Under the hypothesis of (ii), \eqref{R} is an equality so the reverse inclusion holds.\\
(iii) \eqref{Sprime} is equivalent to the following statement ``the set of smooth closed points of $\mathfrak X(\g,\kk)$ is a dense open subset of $\CC$''.
Denote by $$\mathfrak X(J_K(L,\Od))^{sm}=\{(x,y)\in \CC(J_K(L,\Od))\mid \dim Ker\; d_{(x,y)}\mu=\dim \CC(J_K(L,\Od)) \},$$ the set of smooth closed points of $\mathfrak X(\g,\kk)$ belonging  to a given irreducible component $\CC(J_K(L,\Od))$ of $\CC$. 
It is an open set hence  $\mathfrak X(J_K(L,\Od))^{sm}\neq\emptyset$ if and only if there exists 
$(x,y)\in\mathfrak X(J_K(L,\Od))^{sm}$ with $a\in J_K(L,\Od)$. 
Then, Lemma \ref{Po2} (ii) and \eqref{R} show that the existence of such an element ${(x,y)}$ is equivalent to the fact that $\Od$ is involved in rigid pairs of $L$.\\
(iv) follows since \eqref{Sprime} is satisfied for any reduced scheme.
\end{proof}
Note that \cite[Theorem~1.3]{Po}(i) also states that $(0,0)\in \CC^{sing}$ in type 0. 
In $(\sln_3,\sln_2\oplus\tf_1)$, the two irreducible components which are different from $\CC_0$ are isomorphic to a $4$-dimensional vector space. 
In particular, \cite[Theorem~1.3]{Po}(i) can not be generalized to any irreducible component of $\CC$.
However, it is easy to see that $(0,\pp)\cup(\pp,0)\subset T_{(0,0)}(\CC_0)$, hence $T_{(0,0)}(\CC_0)=\pp\times\pp$ and 
$(0,0)\in \CC_0^{sing}$

Since we have computed some lower bounds $d_1(\g,\kk)\leqslant\codim_{\CC_0(\g,\kk)} \CC_0^{irr}(\g,\kk)$ in section \ref{resultdt}, we get a lower bound of $\codim_{\CC_0(\g,\kk)} \CC_0^{sing}(\g,\kk)$ for any symmetric Lie algebra $(\g,\kk)$.  
In particular, one sees that $\codim_{\CC_0} \CC_0^{sing}\geqslant 2$ in all cases.
In \cite[Theorem 3.2]{Pa94}, D.~Panyushev showed that if $(\g,\kk)$ is of maximal rank (\emph{i.e.} $\rk_{sym}(\g,\kk)=\rk \g$) 
then $\mathfrak X(\g,\kk)$ is an irreducible reduced complete intersection. 
This allowed him to prove normality of $\CC(\g,\kk)$ in these cases, cf. \cite[Corollary~4.4]{Pa94}.
Since hypothesis of Proposition \ref{Po4} (ii) are satisfied for these symmetric Lie algebra, we can state the following proposition.

\begin{prop}
Assume that $(\g,\kk)$ is a symmetric Lie algebra of maximal rank. Then,
$$\codim_{\CC(\g,\kk)}\CC(\g,\kk)^{sing}=\left\{\begin{array}{l l}2 &\mbox{ if $\g$ has a simple factor of type G$_2$},\\ 3 & \mbox{ otherwise.} \end{array}\right.$$
\end{prop}

\section*{Appendix: Satake diagrams and irreducibility of $\CC$}
\addcontentsline{toc}{section}{\protect\numberline{}Appendix: Satake diagrams and irreducibility of $\CC$}
\label{appendix}
The table~\ref{table1} recalls the classification of all simple $(\g,\kk)$ when $\g$ is also simple and when $(\g,\kk)$ is simple of type A0.
In the first column we give the type associated to $(\g,\kk)$ in \cite{He}. 
It is of the form EVII where E is the type of the Lie algebra $\g$ and VII is the roman number associated to $(\g,\kk)$ in the classification. 
In the exceptional cases, the \emph{character} is sometimes used to classify symmetric Lie algebras (e.g. in \cite{Dj1}). 
For instance, the notation $E_{7(-25)}$ refers to the unique symmetric Lie algebra $(\g,\kk)$ such that $\g\cong \mathfrak e_7$ and $\dim \pp-\dim \kk=-25$. 
We also give the Satake diagram $S(\g,\kk)$. These diagrams are taken from \cite[p.532]{He}. 
Finally, we gather the results concerning the (ir)reducibility of $\CC(\g,\kk)$ and provide some references for them.


\begin{longtable}{|c|c|c|c|c|c|c|}
\caption{\label{table1} Satake diagrams and irreducibility of $\CC(\g,\kk)$}\\
\hline
Type & Character & $(\g,\kk)$ & Condition & $S(\g,\kk)$& $\CC=\CC_0$ ?& Reference\\
\hline
\endhead
\hline
A0 & / &$(\sln_{n}\oplus\sln_{n},\sln_{n})$ & / & 
\setlength{\unitlength}{0.4mm}
\rule{3mm}{0mm}
\rule{0mm}{10mm}
\begin{picture}(40, 0)
\put(0, 2){\circle{4}}
\put(2, 2){\line(1,0){4}}
\put(8,2){\circle{4}}
\put(10,2){\line(1,0){1}}
\put(12,2){\line(1,0){1}}
\put(14,2){\line(1,0){1}}
\put(16,2){\line(1,0){1}}
\put(18,2){\line(1,0){1}}
\put(20,2){\line(1,0){1}}
\put(23,2){\circle{4}}

\put(0, 4){\vector(0, 1){12}}
\put(0, 16){\vector(0, -1){12}}
\put(8, 4){\vector(0, 1){12}}
\put(8, 16){\vector(0, -1){12}}
\put(23, 4){\vector(0, 1){12}}
\put(23, 16){\vector(0, -1){12}}

\put(0 ,18){\circle{4}}
\put(2 ,18){\line(1,0){4}}
\put(8 ,18){\circle{4}}
\put(10,18){\line(1,0){1}}
\put(12,18){\line(1,0){1}}
\put(14,18){\line(1,0){1}}
\put(16,18){\line(1,0){1}}
\put(18,18){\line(1,0){1}}
\put(20,18){\line(1,0){1}}
\put(23,18){\circle{4}}
\end{picture}
& Yes & \cite{Ri}
\\

\hline
AI & / & $(\sln_n,\so_n)$ & / & \setlength{\unitlength}{0.4mm}
\rule{3mm}{0mm}
\rule{0mm}{4mm}
\begin{picture}(40, 0)
\put(0, 2){\circle{4}}
\put(2, 2){\line(1,0){4}}
\put(8,2){\circle{4}}
\put(10,2){\line(1,0){1}}
\put(12,2){\line(1,0){1}}
\put(14,2){\line(1,0){1}}
\put(16,2){\line(1,0){1}}
\put(18,2){\line(1,0){1}}
\put(20,2){\line(1,0){1}}
\put(23,2){\circle{4}}
\put(25,2){\line(1,0){4}}
\put(31,2){\circle{4}}
\end{picture} & Yes& \cite[3.2]{Pa94}\\

\hline
AII & / &$(\sln_{2n},\spn_{2n})$ & / &
\setlength{\unitlength}{0.4mm}
\rule{3mm}{0mm}
\rule{0mm}{4mm}
\begin{picture}(40, 0)
\put(-8,2){\circle*{4}}
\put(-6,2){\line(1,0){4}}
\put(0, 2){\circle{4}}
\put(2, 2){\line(1,0){4}}
\put(8,2){\circle*{4}}
\put(10,2){\line(1,0){1}}
\put(12,2){\line(1,0){1}}
\put(14,2){\line(1,0){1}}
\put(16,2){\line(1,0){1}}
\put(18,2){\line(1,0){1}}
\put(20,2){\line(1,0){1}}
\put(23,2){\circle*{4}}
\put(25,2){\line(1,0){4}}
\put(31,2){\circle{4}}
\put(33,2){\line(1,0){4}}
\put(39,2){\circle*{4}}
\end{picture} &Yes & \cite[\S3]{Pa04}\\
\hline

\multirow{2}{*}{AIII} & \multirow{2}{*}{/}&\multirow{2}{*}{\begin{tabular}{c}$(\sln_{p+q},\sln_p\oplus\sln_q\oplus\tf_1)$\\ $p\leqslant q$\end{tabular}} & $p=q$ & 
\setlength{\unitlength}{0.4mm}
\rule{3mm}{0mm}
\rule{0mm}{10mm}
\begin{picture}(40, 0)
\put(0, 2){\circle{4}}
\put(2, 2){\line(1,0){4}}
\put(8,2){\circle{4}}
\put(10,2){\line(1,0){1}}
\put(12,2){\line(1,0){1}}
\put(14,2){\line(1,0){1}}
\put(16,2){\line(1,0){1}}
\put(18,2){\line(1,0){1}}
\put(20,2){\line(1,0){1}}
\put(23,2){\circle{4}}
\put(25,2){\line(1,0){6}}
\put(31,2){\line(0,1){6}}
\put(31,10){\circle{4}}
\put(31,12){\line(0,1){6}}
\put(25,18){\line(1,0){6}}

\put(0, 4){\vector(0, 1){12}}
\put(0, 16){\vector(0, -1){12}}
\put(8, 4){\vector(0, 1){12}}
\put(8, 16){\vector(0, -1){12}}
\put(23, 4){\vector(0, 1){12}}
\put(23, 16){\vector(0, -1){12}}

\put(0 ,18){\circle{4}}
\put(2 ,18){\line(1,0){4}}
\put(8 ,18){\circle{4}}
\put(10,18){\line(1,0){1}}
\put(12,18){\line(1,0){1}}
\put(14,18){\line(1,0){1}}
\put(16,18){\line(1,0){1}}
\put(18,18){\line(1,0){1}}
\put(20,18){\line(1,0){1}}
\put(23,18){\circle{4}}
\end{picture}
& Yes & \cite[\S3]{PY}
\\
\cline{4-7}
& & & $p\neq q$ & \setlength{\unitlength}{0.4mm}
\rule{3mm}{0mm}
\rule{0mm}{15mm}
\begin{picture}(40, 0)
\put(0, 2){\circle{4}}
\put(2, 2){\line(1,0){4}}
\put(8,2){\circle{4}}
\put(10,2){\line(1,0){1}}
\put(12,2){\line(1,0){1}}
\put(14,2){\line(1,0){1}}
\put(16,2){\line(1,0){1}}
\put(18,2){\line(1,0){1}}
\put(20,2){\line(1,0){1}}
\put(23,2){\circle{4}}
\put(25,2){\line(1,0){4}}
\put(31,2){\circle*{4}}
\put(31,4){\line(0,1){4}}
\put(31,10){\circle*{4}}

\put(31,12){\line(0,1){1}}
\put(31,14){\line(0,1){1}}
\put(31,16){\line(0,1){1}}
\put(31,18){\line(0,1){1}}
\put(0, 4){\vector(0, 1){23}}
\put(0, 27){\vector(0, -1){23}}
\put(8, 4){\vector(0, 1){23}}
\put(8, 27){\vector(0, -1){23}}
\put(23, 4){\vector(0, 1){23}}
\put(23, 27){\vector(0, -1){23}}

\put(0 ,29){\circle{4}}
\put(2 ,29){\line(1,0){4}}
\put(8 ,29){\circle{4}}
\put(10,29){\line(1,0){1}}
\put(12,29){\line(1,0){1}}
\put(14,29){\line(1,0){1}}
\put(16,29){\line(1,0){1}}
\put(18,29){\line(1,0){1}}
\put(20,29){\line(1,0){1}}
\put(23,29){\circle{4}}
\put(25,29){\line(1,0){4}}
\put(31,29){\circle*{4}}
\put(31,23){\line(0,1){4}}
\put(31,21){\circle*{4}}
\end{picture}
& No & \begin{tabular}{c} \cite[\S4]{PY} or\\ Thm~\ref{mainthm}\end{tabular}
\\
\hline

\multirow{2}{*}{BI} &\multirow{2}{*}{/}& \multirow{2}{*}{\begin{tabular}{c}$(\so_{p+q},\so_p\oplus\so_q)$\\ $p\leqslant q$\end{tabular}} & $p=q-1$ & 
\setlength{\unitlength}{0.4mm}
\rule{3mm}{0mm}
\rule{0mm}{4mm}
\begin{picture}(40, 0)
\put(0, 2){\circle{4}}
\put(2, 2){\line(1,0){4}}
\put(8,2){\circle{4}}
\put(10,2){\line(1,0){1}}
\put(12,2){\line(1,0){1}}
\put(14,2){\line(1,0){1}}
\put(16,2){\line(1,0){1}}
\put(18,2){\line(1,0){1}}
\put(20,2){\line(1,0){1}}
\put(23,2){\circle{4}}
\put(25,2){\line(1,0){4}}
\put(31,2){\circle{4}}
\put(32.5,1){\vector(1,0){5}}
\put(32.5,3){\vector(1,0){5}}
\put(39,2){\circle{4}}
\end{picture}
&Yes & \cite[3.2]{Pa94}
\\
\cline{4-7}
&& & $p\neq q-1$ &
\setlength{\unitlength}{0.4mm}
\rule{3mm}{0mm}
\rule{0mm}{4mm}
\begin{picture}(40, 0)
\put(0, 2){\circle{4}}
\put(2,2){\line(1,0){1}}
\put(4,2){\line(1,0){1}}
\put(6,2){\line(1,0){1}}
\put(8,2){\line(1,0){1}}
\put(11,2){\circle{4}}
\put(13,2){\line(1,0){4}}
\put(19,2){\circle*{4}}
\put(21,2){\line(1,0){1}}
\put(23,2){\line(1,0){1}}
\put(25,2){\line(1,0){1}}
\put(27,2){\line(1,0){1}}
\put(30,2){\circle*{4}}
\put(31.5,1){\vector(1,0){5}}
\put(31.5,3){\vector(1,0){5}}
\put(38,2){\circle*{4}}
\end{picture}& Yes & \cite[\S2]{PY}
\\
\hline

CI & /&$(\spn_{2n}, \gl_{n})$ & / & 
\setlength{\unitlength}{0.4mm}
\rule{3mm}{0mm}
\rule{0mm}{4mm}
\begin{picture}(40, 0)
\put(0, 2){\circle{4}}
\put(2, 2){\line(1,0){4}}
\put(8,2){\circle{4}}
\put(10,2){\line(1,0){1}}
\put(12,2){\line(1,0){1}}
\put(14,2){\line(1,0){1}}
\put(16,2){\line(1,0){1}}
\put(18,2){\line(1,0){1}}
\put(20,2){\line(1,0){1}}
\put(23,2){\circle{4}}
\put(25,2){\line(1,0){4}}
\put(31,2){\circle{4}}
\put(37.5,1){\vector(-1,0){5}}
\put(37.5,3){\vector(-1,0){5}}
\put(39,2){\circle{4}}
\end{picture}&Yes & \cite[3.2]{Pa94}\\

\hline

\multirow{3}{*}{CII} &\multirow{3}{*}{/}& \multirow{3}{*}{\begin{tabular}{c}$(\spn_{2p+2q},\spn_{2p}\oplus\spn_{2q})$\\ $p\leqslant q$\end{tabular}} & $p=q$ & 
\setlength{\unitlength}{0.4mm}
\rule{3mm}{0mm}
\rule{0mm}{4mm}
\begin{picture}(40, 0)
\put(-8, 2){\circle*{4}}
\put(-6, 2){\line(1,0){4}}
\put(0, 2){\circle{4}}
\put(2, 2){\line(1,0){4}}
\put(8,2){\circle*{4}}
\put(10,2){\line(1,0){1}}
\put(12,2){\line(1,0){1}}
\put(14,2){\line(1,0){1}}
\put(16,2){\line(1,0){1}}
\put(18,2){\line(1,0){1}}
\put(20,2){\line(1,0){1}}
\put(23,2){\circle{4}}
\put(25,2){\line(1,0){4}}
\put(31,2){\circle*{4}}
\put(37.5,1){\vector(-1,0){5}}
\put(37.5,3){\vector(-1,0){5}}
\put(39,2){\circle{4}}
\end{picture}& Yes if $p\leqslant 2$ & Prop \ref{irred}\\
\cline{6-7}
& & & & &No  if $p\geqslant 3$&Thm \ref{mainthm}\\

\cline{4-7}
& & & $ p\neq q$ & 
\setlength{\unitlength}{0.4mm}
\rule{3mm}{0mm}
\rule{0mm}{4mm}
\begin{picture}(40, 0)
\put(-8, 2){\circle*{4}}
\put(-6, 2){\line(1,0){4}}
\put(0, 2){\circle{4}}
\put(2, 2){\line(1,0){4}}
\put(8,2){\circle*{4}}
\put(10,2){\line(1,0){1}}
\put(12,2){\line(1,0){1}}
\put(14,2){\line(1,0){1}}
\put(16,2){\line(1,0){1}}
\put(19,2){\circle{4}}
\put(21,2){\line(1,0){4}}
\put(27,2){\circle*{4}}
\put(29,2){\line(1,0){1}}
\put(31,2){\line(1,0){1}}
\put(33,2){\line(1,0){1}}
\put(35,2){\line(1,0){1}}
\put(38,2){\circle*{4}}
\put(44.5,1){\vector(-1,0){5}}
\put(44.5,3){\vector(-1,0){5}}
\put(46,2){\circle*{4}}
\end{picture}\rule{3mm}{0mm} & No & Thm \ref{mainthm}\\
\hline

\multirow{3}{*}{DI} &\multirow{3}{*}{/}& \multirow{3}{*}{\begin{tabular}{c}$(\so_{p+q},\so_p\oplus\so_q)$\\ $p\leqslant q$\end{tabular}} & $p=q$ & 
\setlength{\unitlength}{0.4mm}
\rule{3mm}{0mm}
\rule{0mm}{9mm}
\begin{picture}(40, 0)
\put(0, 10){\circle{4}}
\put(2, 10){\line(1,0){4}}
\put(8,10){\circle{4}}
\put(10,10){\line(1,0){1}}
\put(12,10){\line(1,0){1}}
\put(14,10){\line(1,0){1}}
\put(16,10){\line(1,0){1}}
\put(18,10){\line(1,0){1}}
\put(20,10){\line(1,0){1}}
\put(23,10){\circle{4}}
\put(25,10){\line(1,0){4}}
\put(31,10){\circle{4}}
\put(31,12){\line(0,1){6}}
\put(31,8){\line(0,-1){6}}
\put(31,18){\line(1,0){6}}
\put(31,2){\line(1,0){6}}
\put(39,18){\circle{4}}
\put(39,2){\circle{4}}
\end{picture} & \multirow{3}{*}{Yes} &\multirow{3}{*}{\cite[\S2]{PY}}\\
\cline{4-5}
& & & $p=q-1$ &
\setlength{\unitlength}{0.4mm}
\rule{3mm}{0mm}
\rule{0mm}{9mm}
\begin{picture}(40, 0)
\put(0, 10){\circle{4}}
\put(2, 10){\line(1,0){4}}
\put(8,10){\circle{4}}
\put(10,10){\line(1,0){1}}
\put(12,10){\line(1,0){1}}
\put(14,10){\line(1,0){1}}
\put(16,10){\line(1,0){1}}
\put(18,10){\line(1,0){1}}
\put(20,10){\line(1,0){1}}
\put(23,10){\circle{4}}
\put(25,10){\line(1,0){4}}
\put(31,10){\circle{4}}
\put(31,12){\line(0,1){6}}
\put(31,8){\line(0,-1){6}}
\put(31,18){\line(1,0){6}}
\put(31,2){\line(1,0){6}}
\put(39,18){\circle{4}}
\put(39,2){\circle{4}}
\put(39,4){\vector(0,1){12}}
\put(39,16){\vector(0,-1){12}}
\end{picture}& &\\
\cline{4-5}
& & & $p<q-1$ &

\setlength{\unitlength}{0.4mm}
\rule{3mm}{0mm}
\rule{0mm}{9mm}
\begin{picture}(40, 0)
\put(1, 10){\circle{4}}
\put(3,10){\line(1,0){1}}
\put(5,10){\line(1,0){1}}
\put(7,10){\line(1,0){1}}
\put(9,10){\line(1,0){1}}
\put(12,10){\circle{4}}
\put(14, 10){\line(1,0){4}}
\put(20,10){\circle*{4}}
\put(22,10){\line(1,0){1}}
\put(24,10){\line(1,0){1}}
\put(26,10){\line(1,0){1}}
\put(28,10){\line(1,0){1}}
\put(31,10){\circle*{4}}
\put(31,12){\line(0,1){6}}
\put(31,8){\line(0,-1){6}}
\put(31,18){\line(1,0){6}}
\put(31,2){\line(1,0){6}}
\put(39,18){\circle*{4}}
\put(39,2){\circle*{4}}
\end{picture}&&\\

\hline

\multirow{2}{*}{DIII} &\multirow{2}{*}{/}& \multirow{2}{*}{\begin{tabular}{c}$(\so_{2n},\gl_n)$\\ $n\geqslant 4$ \end{tabular}} & $n$ odd & 
\setlength{\unitlength}{0.4mm}
\rule{3mm}{0mm}
\rule{0mm}{9mm}
\begin{picture}(40, 0)
\put(-8, 10){\circle*{4}}
\put(-6, 10){\line(1,0){4}}
\put(0, 10){\circle{4}}
\put(2, 10){\line(1,0){4}}
\put(8,10){\circle*{4}}
\put(10,10){\line(1,0){1}}
\put(12,10){\line(1,0){1}}
\put(14,10){\line(1,0){1}}
\put(16,10){\line(1,0){1}}
\put(18,10){\line(1,0){1}}
\put(20,10){\line(1,0){1}}
\put(23,10){\circle{4}}
\put(25,10){\line(1,0){4}}
\put(31,10){\circle*{4}}
\put(31,12){\line(0,1){6}}
\put(31,8){\line(0,-1){6}}
\put(31,18){\line(1,0){6}}
\put(31,2){\line(1,0){6}}
\put(39,18){\circle{4}}
\put(39,2){\circle{4}}
\put(39,4){\vector(0,1){12}}
\put(39,16){\vector(0,-1){12}}
\end{picture}& No & \begin{tabular}{c} \cite[\S4]{PY} or \\ Thm \ref{mainthm}\end{tabular}\\
\cline{4-7} 
& & & $n$ even & \setlength{\unitlength}{0.4mm}
\rule{3mm}{0mm}
\rule{0mm}{9mm}
\begin{picture}(40, 0)
\put(-8, 10){\circle*{4}}
\put(-6, 10){\line(1,0){4}}
\put(0, 10){\circle{4}}
\put(2, 10){\line(1,0){4}}
\put(8,10){\circle*{4}}
\put(10,10){\line(1,0){1}}
\put(12,10){\line(1,0){1}}
\put(14,10){\line(1,0){1}}
\put(16,10){\line(1,0){1}}
\put(18,10){\line(1,0){1}}
\put(20,10){\line(1,0){1}}
\put(23,10){\circle{4}}
\put(23,12){\line(0,1){6}}
\put(23,8){\line(0,-1){6}}
\put(23,18){\line(1,0){6}}
\put(23,2){\line(1,0){6}}
\put(31,18){\circle*{4}}
\put(31,2){\circle{4}}
\end{picture}& No & Thm \ref{mainthm}
\\
\hline

EI & E$_{6(6)}$ &$(\mathfrak e_6, \spn_8)$& / &
\setlength{\unitlength}{0.4mm}
\rule{5mm}{0mm}
\rule{0mm}{7mm}
\begin{picture}(40, 0)
\put(-8,10){\circle{4}}
\put(-6,10){\line(1,0){4}}
\put(0, 10){\circle{4}}
\put(2, 10){\line(1,0){4}}
\put(8,10){\circle{4}}
\put(10,10){\line(1,0){4}}
\put(8,8){\line(0,-1){4}}
\put(8,2){\circle{4}}
\put(16,10){\circle{4}}
\put(18,10){\line(1,0){4}}
\put(24, 10){\circle{4}}
\end{picture} & Yes & \cite[3.2]{Pa94}\\

\hline
EII & E$_{6(2)}$&$(\mathfrak e_6, \sln_6\oplus\sln_2)$ &/ & \setlength{\unitlength}{0.4mm}
\rule{5mm}{0mm}
\rule{0mm}{11mm}
\begin{picture}(40, 0)
\put(-8,21){\vector(0,-1){9}}
\put(24,21){\vector(0,-1){9}}
\put(-8,21){\line(1,0){32}}
\put(0,18){\vector(0,-1){6}}
\put(16,18){\vector(0,-1){6}}
\put(0,18){\line(1,0){16}}

\put(-8,10){\circle{4}}
\put(-6,10){\line(1,0){4}}
\put(0, 10){\circle{4}}
\put(2, 10){\line(1,0){4}}
\put(8,10){\circle{4}}
\put(10,10){\line(1,0){4}}
\put(8,8){\line(0,-1){4}}
\put(8,2){\circle{4}}
\put(16,10){\circle{4}}
\put(18,10){\line(1,0){4}}
\put(24, 10){\circle{4}}
\end{picture} & Yes & \cite[\S3]{PY}\\
\hline
EIII & E$_{6(-14)}$&$(\mathfrak e_6, \so_{10}\oplus\tf_1)$ & / & \setlength{\unitlength}{0.4mm}
\rule{5mm}{0mm}
\rule{0mm}{9mm}
\begin{picture}(40, 0)
\put(-8,18){\vector(0,-1){6}}
\put(24,18){\vector(0,-1){6}}
\put(-8,18){\line(1,0){32}}

\put(-8,10){\circle{4}}
\put(-6,10){\line(1,0){4}}
\put(0, 10){\circle*{4}}
\put(2, 10){\line(1,0){4}}
\put(8,10){\circle*{4}}
\put(10,10){\line(1,0){4}}
\put(8,8){\line(0,-1){4}}
\put(8,2){\circle{4}}
\put(16,10){\circle*{4}}
\put(18,10){\line(1,0){4}}
\put(24, 10){\circle{4}}
\end{picture}& No & \begin{tabular}{c} \cite[\S4]{PY} or \\ Thm \ref{mainthm}\end{tabular}\\
\hline
EIV & E$_{6(-26)}$&$(\mathfrak e_6, \mathfrak f_4)$ &/ &  \setlength{\unitlength}{0.4mm}
\rule{5mm}{0mm}
\rule{0mm}{7mm}
\begin{picture}(40, 0)
\put(-8,10){\circle{4}}
\put(-6,10){\line(1,0){4}}
\put(0, 10){\circle*{4}}
\put(2, 10){\line(1,0){4}}
\put(8,10){\circle*{4}}
\put(10,10){\line(1,0){4}}
\put(8,8){\line(0,-1){4}}
\put(8,2){\circle*{4}}
\put(16,10){\circle*{4}}
\put(18,10){\line(1,0){4}}
\put(24, 10){\circle{4}}
\end{picture} & Yes &\cite[\S3]{Pa04}\\
\hline
EV & E$_{7(7)}$& $(\mathfrak e_7, \sln_9)$& / & \setlength{\unitlength}{0.4mm}
\rule{5mm}{0mm}
\rule{0mm}{7mm}
\begin{picture}(40, 0)
\put(-8,10){\circle{4}}
\put(-6,10){\line(1,0){4}}
\put(0, 10){\circle{4}}
\put(2, 10){\line(1,0){4}}
\put(8,10){\circle{4}}
\put(10,10){\line(1,0){4}}
\put(8,8){\line(0,-1){4}}
\put(8,2){\circle{4}}
\put(16,10){\circle{4}}
\put(18,10){\line(1,0){4}}
\put(24, 10){\circle{4}}
\put(26, 10){\line(1,0){4}}
\put(32,10){\circle{4}}
\end{picture}& Yes & \cite[3.2]{Pa94}\\
\hline
EVI & E$_{7(-5)}$&$(\mathfrak e_7, \so_{12}\oplus\sln_2)$& / & \setlength{\unitlength}{0.4mm}
\rule{5mm}{0mm}
\rule{0mm}{7mm}
\begin{picture}(40, 0)
\put(-8,10){\circle{4}}
\put(-6,10){\line(1,0){4}}
\put(0, 10){\circle{4}}
\put(2, 10){\line(1,0){4}}
\put(8,10){\circle{4}}
\put(10,10){\line(1,0){4}}
\put(8,8){\line(0,-1){4}}
\put(8,2){\circle*{4}}
\put(16,10){\circle*{4}}
\put(18,10){\line(1,0){4}}
\put(24, 10){\circle{4}}
\put(26, 10){\line(1,0){4}}
\put(32,10){\circle*{4}}
\end{picture}& No & \begin{tabular}{c}\cite[\S5]{PY} or\\ Thm \ref{mainthm}\end{tabular}\\
\hline
EVII & E$_{7(-25)}$&$(\mathfrak e_7, \mathfrak e_6\oplus\tf_1)$& / & \setlength{\unitlength}{0.4mm}
\rule{5mm}{0mm}
\rule{0mm}{7mm}
\begin{picture}(40, 0)
\put(-8,10){\circle{4}}
\put(-6,10){\line(1,0){4}}
\put(0, 10){\circle*{4}}
\put(2, 10){\line(1,0){4}}
\put(8,10){\circle*{4}}
\put(10,10){\line(1,0){4}}
\put(8,8){\line(0,-1){4}}
\put(8,2){\circle*{4}}
\put(16,10){\circle*{4}}
\put(18,10){\line(1,0){4}}
\put(24, 10){\circle{4}}
\put(26, 10){\line(1,0){4}}
\put(32,10){\circle{4}}
\end{picture}& No & Thm \ref{mainthm}\\
\hline
EVIII & E$_{8(8)}$&$(\mathfrak e_8, \so_{16})$&/ & \setlength{\unitlength}{0.4mm}
\rule{5mm}{0mm}
\rule{0mm}{7mm}
\begin{picture}(40, 0)
\put(-8,10){\circle{4}}
\put(-6,10){\line(1,0){4}}
\put(0, 10){\circle{4}}
\put(2, 10){\line(1,0){4}}
\put(8,10){\circle{4}}
\put(10,10){\line(1,0){4}}
\put(8,8){\line(0,-1){4}}
\put(8,2){\circle{4}}
\put(16,10){\circle{4}}
\put(18,10){\line(1,0){4}}
\put(24, 10){\circle{4}}
\put(26, 10){\line(1,0){4}}
\put(32,10){\circle{4}}
\put(34, 10){\line(1,0){4}}
\put(40,10){\circle{4}}
\end{picture}& Yes & \cite[3.2]{Pa94}\\
\hline
EIX & E$_{8(-24)}$ &$(\mathfrak e_8, \mathfrak e_7\oplus\sln_2)$&/ & \setlength{\unitlength}{0.4mm}
\rule{5mm}{0mm}
\rule{0mm}{7mm}
\begin{picture}(40, 0)
\put(-8,10){\circle{4}}
\put(-6,10){\line(1,0){4}}
\put(0, 10){\circle*{4}}
\put(2, 10){\line(1,0){4}}
\put(8,10){\circle*{4}}
\put(10,10){\line(1,0){4}}
\put(8,8){\line(0,-1){4}}
\put(8,2){\circle*{4}}
\put(16,10){\circle*{4}}
\put(18,10){\line(1,0){4}}
\put(24, 10){\circle{4}}
\put(26, 10){\line(1,0){4}}
\put(32,10){\circle{4}}
\put(34, 10){\line(1,0){4}}
\put(40,10){\circle{4}}
\end{picture}& No & \begin{tabular}{c}\cite[\S5]{PY} or\\ Thm \ref{mainthm}\end{tabular}\\
\hline
FI & F$_{4(4)}$&$(\mathfrak f_4, \spn_6\oplus\sln_2)$&/ &
\setlength{\unitlength}{0.4mm}
\rule{5mm}{0mm}
\rule{0mm}{4mm}
\begin{picture}(40, 0)
\put(3,2){\circle{4}}
\put(5,2){\line(1,0){4}}
\put(11,2){\circle{4}}
\put(12.5,1){\vector(1,0){5}}
\put(12.5,3){\vector(1,0){5}}
\put(19,2){\circle{4}}
\put(21,2){\line(1,0){4}}
\put(27,2){\circle{4}}
\end{picture}& Yes & \cite{Pa94}\\
\hline
FII & F$_{4(-20)}$&$(\mathfrak f_4, \so_9)$& /& 
\setlength{\unitlength}{0.4mm}
\rule{5mm}{0mm}
\rule{0mm}{4mm}
\begin{picture}(40, 0)
\put(3,2){\circle*{4}}
\put(5,2){\line(1,0){4}}
\put(11,2){\circle*{4}}
\put(12.5,1){\vector(1,0){5}}
\put(12.5,3){\vector(1,0){5}}
\put(19,2){\circle*{4}}
\put(21,2){\line(1,0){4}}
\put(27,2){\circle{4}}
\end{picture}& No & \begin{tabular}{c}\cite[\S4]{Pa04} or\\ Thm \ref{mainthm}\end{tabular} \\
\hline
GI & G$_{2(2)}$&$(\mathfrak g_2, \sln_2\oplus\sln_2)$& /& 
\setlength{\unitlength}{0.4mm}
\rule{5mm}{0mm}
\rule{0mm}{4mm}
\begin{picture}(40, 0)
\put(11,2){\circle{4}}
\put(12,0){\vector(1,0){10}}
\put(13,2){\line(1,0){8}}
\put(12,4){\vector(1,0){10}}
\put(23,2){\circle{4}}\end{picture}& Yes &\cite{Pa94}\\
\hline

\end{longtable}

\begin{remark}
Applying one of the following two arguments, one can obtain all the known cases of irreducibility of $\CC(\g,\kk)$. 
\begin{itemize}
\item The first one is Proposition~\ref{irred} which is basically Richardson's proof of irreducibility in type 0. 
In order to use it, we need to prove that there are no non-zero rigid $\pp$-distinguished orbit in each $\pp$-Levi of $(\g,\kk)$.
\item The second argument is provided in \cite[\S3]{PY}. It can be applied to any symmetric Lie algebra whose Satake diagrams has only white nodes and at least one non-arrowed white node.
\end{itemize}
\end{remark}



\begin{thebibliography}{999}
    \bibitem[An]{An} L.~V.~Antonyan, On classification of homogeneous elements of $\Z_2$-graded semisimple Lie algebras, 
      \emph{Vestnik Moskov. Univ. Ser. I Mat. Mekh.}, {\bf37}
      (1982), 29-34.
    \bibitem[Ar]{Ar} S.~Araki, On root systems and an
      infinitesimal classification of irreducible symmetric
      spaces, \emph{J.~Math. Osaka City Univ.}, {\bf13}
      (1962), 1-34.
    \bibitem[BC]{BC} P.~Bala and R.~W.~Carter, Classes of
      unipotent elements in simple algebraic groups.~II,
      \emph{Math. Proc. Camb. Phil. Soc}, {\bf 80} (1976),
      1-18.
    \bibitem[Bo]{Bo} W.~Borho, \"Uber schichten halbeinfacher
      Lie-algebren, \emph{Invent. Math.}, {\bf65} (1981),
     283-317.
    \bibitem [Bu1] {Bu1} M.~Bulois, Composantes
      irr\'eductibles de la vari\'et\'e commutante
      nilpotente d'une alg\`ebre de Lie sym\'etrique
      semi-simple [In French], \emph{Annales de l'institut Fourier},
      {\bf 59} (2009), 37-80.
    \bibitem [Bu2] {Bu2} M.~Bulois, Sheets of semisimple
      symmetric Lie algebras and Slodowy slices, preprint, 
      \url{http://arxiv.org/abs/0906.3881} (2010), to appear in \emph{Journal of Lie Theory}.
    \bibitem[CM]{CM} D.~H.~Collingwood and W.~M.~McGovern,
      \emph{Nilpotent orbits in semisimple Lie algebras},
      Van Nostrand Reinhold Mathematics Series, New York,
      1993.
    \bibitem [Dj1] {Dj1} D.~Z.~Djokovic, Classification of
      nilpotent elements in simple exceptional real Lie algebras
      of inner type and description of their centralizers,
      \emph{J. Algebra}, {\bf112} (1988), 503-524.
    \bibitem [Dj2] {Dj2} D.~Z.~Djokovic, Classification of
      nilpotent elements in simple real Lie algebras
      $E_{6(6)}$ and $E_{6(-26)}$ and description of their
      centralizers, \emph{J. Algebra}, {\bf116} (1988), 196-207.
    \bibitem[EP1]{EP1} A.~G.~Elashvili and D.~Panyushev, Towards a classification of principal nilpotent pairs, \cite[Appendix]{Gi}.
    \bibitem[EP2]{EP2} A.~G.~Elashvili and D.~Panyushev, A classification of the principal nilpotent pairs in simple Lie algebras and related problems, 
    \emph{J. London Math. Soc. (2)},  {\bf 63} (2001), 299-318. 
    \bibitem [Gi]{Gi} V.~Ginzburg, Principal nilpotent pairs in a semisimple Lie algebra I., 
      \emph{Invent. Math.}, {\bf 140} (2000), 511-561. 
    \bibitem [GW]{GW} R.~Goodman and N.~R.~Wallach, An
      algebraic group approach to compact symmetric spaces,
      \url{http://www.math.rutgers.edu/~goodman/pub/symspace.pdf}.
      \emph{In}: \emph{Representations and Invariants of the
      Classical Groups}, Cambridge University Press, 1998.
    \bibitem[Ha]{Ha}  R.~Hartshorne, \emph{Algebraic geometry}, 
     Graduate Texts in Mathematics, {\bf 52}, Springer-Verlag, 1977.
    \bibitem[He]{He} S.~Helgason, \emph{Differential
        geometry, Lie groups, and symmetric spaces}, Pure
      and applied mathematics, Academic press, 1978.
    \bibitem [JN] {JN} S.~G.~Jackson and A.~G.~Noel,
      Prehomogeneous spaces associated with nilpotent orbits,
      \url{http://www.math.umb.edu/~anoel/publications/tables} (2005).
    \bibitem[KR]{KR} B.~Kostant and S.~Rallis, Orbits and
      representations associated with symmetric spaces,
      \emph{Amer. J. Math.}, {\bf93} (1971), 753-809.
    \bibitem[LS]{LS} G.~Lusztig and N.~Spaltenstein,
      Induced unipotent classes,
     \emph{J. London Math. Soc. (2)}, {\bf19} (1979), 41-52.
    \bibitem[Oh1]{Oh1} T.~Ohta, The singularities of the
      closure of nilpotent orbits in certain symmetric
      pairs, \emph{Tohoku Math. J.}, {\bf38} (1986),
      441-468.
    \bibitem[Oh2]{Oh2} T.~Ohta, The closure of nilpotent
      orbits in the classical symmetric pairs and their
      singularities, \emph{Tohoku Math. J.}, {\bf43} (1991),
      161-211.
    \bibitem [Pa1]{Pa94} D.~I.~Panyushev, The Jacobian
      modules of a representation of a Lie algebra and
      geometry of commuting varieties, \emph{Compositio
        Math.}, {\bf94} (1994), 181-199.
    \bibitem [Pa2]{Pa00} D.~I.~Panyushev, Nilpotent pairs in semisimple Lie algebras 
        and their characteristics, \emph{Internat. Math. Res. Notices}, {\bf2000} , 1-21.
    \bibitem [Pa3]{Pa01} D.~I.~Panyushev, Nilpotent pairs, dual pairs and sheets, 
    \emph{J.~of Algebra}, {\bf240} (2001), 635-664.
    \bibitem [Pa4] {Pa04} D.~I.~Panyushev, On the
      irreducibility of commuting varieties associated with
      involutions of simple Lie algebras,
      \emph{Func. Anal. Appl.}, {\bf38} (2004), 38-44.
    \bibitem [PY] {PY} D.~I.~Panyushev and O.~Yakimova,
      Symmetric pairs and associated commuting varieties,
      \emph{Math. Proc. Cambr. Phil. Soc.}, {\bf143} (2007),
      307-321.
    \bibitem [Po]{Po} V.~L.~Popov, Irregular and singular loci of commuting varieties,
      \emph{Transform. Groups}, {\bf 13} (2008),  819-837.
    \bibitem [PT]{PT} V.~L.~Popov and E.~A.~Tevelev,
      Self-dual projective algebraic varieties associated
      with symmetric spaces. \emph{In}: Algebraic
      transformation groups and algebraic varieties,
      \emph{Enc. Math. Sci.}, {\bf132}, Springer-Verlag,
      2004, 131-167.
    \bibitem [Pr]{Pr} A.~Premet, Nilpotent commuting varieties of reductive Lie algebras,
      \emph{Invent. Math.}, {\bf 154} (2003), 653-683. 
    \bibitem [Ri] {Ri} R.~W.~Richardson, Commuting
      varieties of semisimple Lie algebras and algebraic
      groups, \emph{Compositio Math.}, {\bf38} (1979),
      311-327.
    \bibitem [SY] {SY} H.~Sabourin and R.~W.~T.~Yu, On the
      irreducibility of the commuting variety of the
      symmetric pair $(\so_{p+2},\so_{p}\times \so_{2})$,
      \emph{J. Lie Theory}, {\bf16} (2006), 57-65.
    \bibitem[TY]{TY} P.~Tauvel and R.~W.~T.~Yu, \emph{Lie
        algebras and algebraic groups}, Springer Monographs
      in Mathematics, Springer-Verlag, 2005.
    \bibitem[Yu]{Yu} R.~W.~T.~Yu, Centralizers of distinguished nilpotent pairs and related problems,
      \emph{J.~of Algebra}, {\bf252} (2002), 167-194.
  \end{thebibliography}
\end{document}